\documentclass[12pt]{article}

\usepackage{amssymb}

\usepackage{common}

\title{On the existence property over a predicate}

\author{Alexander Usvyatsov\thanks{ In memory of Zo\'{e} Chatzidakis. \newline \quad    
The author thanks the Austrian Science Foundation (FWF), projects P33895 and  P33420, for supporting this research. 
}}
 
\newcommand{\Addresses}{{
  \bigskip
  \footnotesize

Alexander Usvyatsov, \textsc{Institut f\"{u}r Diskrete Mathematik und Geometrie,
TU Wien,
1040, Vienna, Austria}

}} 
 
\newtheorem{theorem}{Theorem}[section]

\newtheorem{definition}[theorem]{Definition}
\newtheorem{example}[theorem]{Example}
\newtheorem{lem}[theorem]{Lemma}
\newtheorem{obs}[theorem]{Observation}
\newtheorem{co}[theorem]{Corollary}

\newtheorem{hyp}[theorem]{Hypothesis}
\newtheorem{remark}[theorem]{Remark}

\newtheorem{con}[theorem]{Conjecture}
\newtheorem{conv}[theorem]{Convention}
\newtheorem{dsc}[theorem]{Discussion}
\newtheorem{qst}[theorem]{Question}
\newtheorem{ft}[theorem]{Fact}
\renewenvironment{proof}{\noindent {\em Proof:}}{\hspace*{1cm}
        \hspace*{\fill}$\rule{1.2ex}{1.4ex}$\medskip} 
\newenvironment{re}{\begin{remark}\rm}{\end{remark}}
\newenvironment{ex}{\begin{example}\rm}{\end{example}} 
\newenvironment{de}{\begin{definition}\rm}{\end{definition}}
\newenvironment{fact}{\begin{ft}\rm}{\end{ft}}

\date{}

\begin{document}

%

\maketitle

\abstract{We prove that in a countable theory $T$ fully stable over a predicate $P$, any complete set $A$ has the existence property. This means that $A$ can be extended to a model of $T$ without changing the $P$-part. In particular, $T$ has the Gaifman property: any model of $P$ occurs as the $P$-part of some model of $T$.} This generalizes results of Lachlan (on stable theories), Hodges (on relatively categorical abelian groups), and Afshordel (on difference fields of characteristic 0).

\section{Introduction}

\subsection{Motivation and Summary}

This paper is concerned with Classification Theory over a predicate, following previous investigations, such as \cite{Pil-cat,PiSh130,Sh234,ShUs322a,usvyatsov2024stability}. In this framework, we are given a complete first order theory $T$, where $P$ is a distinguished unary predicate in its vocabulary. In order to make more sense of the questions below, we make a few additional assumptions (such as: $T$ has quantifier elimination, and $P$ is stably embedded). We discuss these assumptions and their validity in more detail in section 2, Hypothesis \ref{asm:1}. The goal is understanding the complexity and the properties of the class of models of $T$ with a prescribed $P$-part. More precisely, let $\cC$ be the monster model of $T$, and let $T^P$ be the theory of its $P$-part. One question that one can ask is: how complex can be the class of models of $T$ with a given $P$-part? A different (but related) natural question is: is the class of models of $T$ with a prescribed $P$-part always non-empty? And more generally, under what conditions, given $N\models T^P$, does there exist a model of $T$ whose $P$-part equals $N$? Our thesis is  that \emph{stability over $P$} (Definition  \ref{dfn:startypes})  should play an important role in this discussion. Here we provide some evidence for this. 

The results in this paper are partially motivated by the following example. Consider $T = ACFA_0$, the theory of algebraically closed fields of characteristic $0$ with a generic automorphism $\sigma$. Let $\cC$ be a monster model of $T$, and let $P$ be a new unary predicate symbol that is interpreted as $Fix(\sigma)$, the fixed field of $\sigma$. The theory $T$ is simple unstable, but it was already noted by Chatzidakis and Hrushovski in \cite{Chatzidakis1999ModelTO} that in some strong sense, the only source of instability lies within the fixed field (this idea can be formalized in the notion of\emph{ stability over $P$}, which we will define and discuss in detail later).  This observation has led to the realization that stability theoretic techniques can be applied to the class of models of $ACFA_0$ ``over'' $P = Fix(\sigma)$. For example, in \cite{Chatzidakis2020RemarksAT}, Chatzidakis has used inspiration from stability theory in order to construct prime models in the class of $\lam$-saturated models of $T$ with a prescribed $\lam$-saturated (pseudo-finite) fixed field. The author has generalized this result to theories fully stable over $P$ in \cite{usvyatsov2024stability}.

More relevant to our present work here are the results of Afshordel  \cite{MR3343525}. He has shown that every pseudo-finite fields occurs  as the fixed field of some model of $ACFA$, and, moreover, any difference field $A$ whose fixed field $k$ is pseudo-finite, embeds into a model of $ACFA$ whose fixed field is $k$. 

Motivated by these results, we show that similar strong existence properties hold under the assumption of stability over $P$. This generalizes Afshordel's theorem (in characteristic 0) to an arbitrary countable theory $T$, fully stable over a predicate $P$. It also generalizes Lachlan's results  \cite{Lachlan1972} that imply full existence for the case of $T$ countable and stable (and $P$ any unary predicate in $T$), and some of the work of Hodges on Gaifman's Conjecture for abelian groups, e.g., \cite{Hod-cat1}.

We explain the connection with Lachlan's and Hodges' work in Remark \ref{re:lachlan}, Corollary \ref{co:lachlan}, and Corollary \ref{co:hodges}.


In section 8, we discuss $ACFA_0$, and show that indeed  (any completion of) $ACFA_0$ is fully stable over $P$. This confirms that this is an example of our framework, and all techniques developed here apply to it.

 In the last section we briefly discuss the assumptions that Shelah makes in \cite{Sh234} (under which he proves the existence property for a large family of complete sets),  introduce the notion of nulldimensionality of a theory $T$ over a predicate $P$, and connect it to categoricity over $P$. 

%

\subsection{Background}

Let $T$ be a complete first order theory, $P$ a distinguished predicate in its vocabulary. In \cite{gaifman}, Gaifman conjectured that any countable theory $T$ categorical over $P$ (has at most one model with any given prescribed $P$ -part, up to isomorphism over $P$) has the following property: every model of $T^P$ (the theory of $P$) occurs as the $P$-part of some model of $T$. In other words, if $T$ has at most one model with a prescribed $P$-part, it has \emph{exactly} one such model for each model of $T^P$. Gaifman proved this statement for the case that $T$ is rigid over $P$. Specific non-rigid examples (such as $T$ is a theory of an abelian group, $P$ is a subgroup) were investigated by Hodges, e.g., \cite{Hod-cat1, Hod-cat3}.

Shelah established the full statement of Gaifman's Conjecture in \cite{Sh234} modulo some set-theoretic issues that have not been yet completely resolved: specifically, it is currently our understanding that the results in \cite{Sh234} hold for theories that are \emph{provably} categorical over $P$, i.e., the categoricity of $T$ over $P$ is provable in ZFC, or at least the categoricity is preserved in all forcing extensions (we discuss Shelah's theorem in more detail in section 9). In order to derive the statement above, Shelah actually needed to establish a much stronger result. Specifically, let us say that a subset $A$ of the monster model $\cC$ of $T$ has \emph{the existence property over $P$} (or simply the existence property) if there exists $M \models T$  with $A \subseteq M$ and $P^M = P^A$.  In order for $A$ to have the existence property, it is absolutely essential to assume that $P^A$ is ``closed'' within $\cC$ in certain ways; we will summarize this requirement in the definition of \emph{complete sets} later in section 2.1. In order to establish Gaifman's Conjecture, Shelah proved the existence property for a whole class of complete sets of a certain form called ``$\mathcal P^-(n)$-systems'' (and not only over models of $T^P$). 

Given these results, it is natural to ask some  general questions. For example: under what (weaker) assumptions on $T$ does the conclusion of Gaifman's Conjecture hold? One may also want to expand the question to bigger categories of complete sets. For instance, we can ask whether (and when) the existence property is true for all complete sets, or, at least, all complete sets whose $P$-part is ``nice enough'' (e.g., somewhat saturated). The former question is addressed in a recent preprint \cite{ShUs322b}. Here we focus on the latter one. 

\smallskip

We will say that $T$ has the \emph{full existence property over $P$}, or \emph{full existence} for short, if all complete subsets of the monster model of $T$ have the existence property (see Definition \ref{de:fullexist}). One natural question is:

\begin{qst}\label{qst:1}
 Under which assumptions does $T$ have the full existence property?
\end{qst}

Sometimes it is natural to restrict oneself to a certain subclass of complete sets. For example, one may wonder when $T$ has the existence property for the class of all complete sets with a specific $P$-part. Alternatively, one can ask about complete sets whose $P$-part is of a certain form (e.g., is somewhat saturated). 

\begin{qst}\label{qst:2}
Let $N$ be a model of the theory of $P$. Under which conditions on $N$ does $T$ have the existence property for all complete sets $A$ with $P^A = N$? 
\end{qst}
  
%

\bigskip

As we have already mentioned, the existence property (or some variant thereof) is known to be true in many natural examples (that are not categorical over $P$). For example, if $T$ is stable and countable, and $P$ is any unary predicate it its vocabulary, Lachlan's results \cite{Lachlan1972} imply the full existence property. 
If $T$ is the theory of a vector space $V$ over a field $F$ (where $P$ is interpreted as the field), then  $T$ (which will normally be unstable) also has the full existence property. 

A more interesting example (that we have already mentioned earlier) arises from $ACFA$, the theory  of algebraically closed fields with a generic automorphism. 
Afshordel's results in \cite{MR3343525} imply that any completion $T$  of $ACFA$ has full existence over the fixed field $P = Fix(\sigma)$. The particular case of Afshordel's theorem for $ACFA_0$ (the theory of algebraically closed fields of characteristic $0$ with a generic automorphism) has inspired our investigation, since it falls in the category of theories stable over $P$. The main theorem in our paper confirms that any countable theory fully stable over $P$ (satisfying a few ``reasonable'' assumptions; see Hypothesis \ref{asm:1}) has full existence.
It also provides a general procedure for constructing locally constructible models over complete sets, assuming stability over $P$.  

\smallskip

Another (more complex, and much less understood) example for which we know (a version of) the existence property is the class of exponentially closed fields, $ECF$, studied by Kirby and Zilber \cite{KiZil}. 
This class is elementary under the assumption  of CIT, Zilber's Conjecture known as the Conjecture of Intersection of Tori with Varieties (in fact, CIT is equivalent to $ECF$ being an elementary class \cite{KiZil}). Therefore, under CIT, its theory fits in the context of model theory over $P$, with $P$ interpreted as the kernel of the exponential function, to which we will just refer below as ``the kernel''. 

Kirby and Zilber have shown that every $\aleph_0$-saturated model of $Th(\mathbb Z)$, the theory of integers (as a ring), occurs as the multiplicative stabilizer of the kernel of some exponentially closed field $F$. Moreover, $F$ is saturated over its kernel, hence embeds exponentially closed fields $K$ with $Ker(K) = Ker(F)$ of cardinality $|K| \le |F|$. It follows that the existence property (over $P= Ker(exp)$) is true for subsets of exponentially closed fields with an $\aleph_0$-saturated  kernel.

 The theorem above also implies that $ECF$ is (in a certain sense) stable (even superstable) over $P$, hence (under CIT) our discussion is relevant to this theory. However, we do not know whether $ECF$ is \emph{fully} stable over $P$ (what is missing is understanding of complete sets that are not models of $T$, and of relevant types over these sets), so it is not currently clear whether all of the results in this paper apply to it. 

\smallskip

Finally, some general results regarding Question \ref{qst:2} have also been established. For example, Proposition 4.13 in \cite{ShUs322a} says that \emph{any} theory (that satisfies our basic assumptions, Hypothesis \ref{asm:1}, of course) has the existence property over saturated models of the theory of $P$.

\bigskip

Here we address the two questions stated above (Questions \ref{qst:1}, \ref{qst:2}) and investigate the existence property under the assumption of stability over $P$. Specificaly, we prove the following: if $T$ is stable over a certain $N \models T^P$, then every complete set $A$ with $P^A = N$ has the existence property (Theorem \ref{th:primary}). In particular, if $T$ is fully stable over $P$ (all complete sets are stable), then it has the full existence property. 

The proof of Theorem \ref{th:primary} does not require the full strength of the assumption stated above. For example, as we point out in Remark \ref{re:stab_unions}, it is enough to assume that stability (over $P$) of complete sets is preserved under taking unions of arbitrary increasing continuous chains. It would be interesting to investigate these properties and connections further.

\bigskip

The paper is organized as follows. Section 2 contains all the necessary preliminaries, including the assumptions that we make on $T$ and $P$, the definitions of complete sets, stability over $P$, and basic properties of these notions. In section 3 we recall the rank function relevant for stability over $P$. In section 4 we revisit the notion of complete sets, and prove a few simple new characterizations, that  come in handy in our later discussion. Section 5 is devoted to the discussion of minimal types (with respect to the rank discussed in section 3). Section 6 contains definitions and basic properties of the notions of isolation and constructibility relevant for this work. Section 7 is devoted to the proof of the main result, Theorem  \ref{th:primary}.

Section 8 addresses the following natural question: is the assumption of full stability too strong? So far (with the exception of one recent paper of the author \cite{usvyatsov2024stability}) only stability over $P$ of models and systems of models (e.g., \cite{Sh234, ShUs322a, ShUs322b}) has been assumed in the program of Classification Theory over $P$. One reason for this is that negating these assumptions leads to non-structure results \cite{Sh234}, which, according to the Classification Theory guidelines, makes them  reasonable (and justified). At the same time, our results indicate that full stability can be incredibly useful  for various applications. In is therefore advisable to make sure that some natural and motivating examples are indeed fully stable. Here we confirm this for (any completion of) $ACFA_0$. We are positive that there are many more such natural examples. 

Finally, section 9 is devoted to a brief discussion of Gaifman's Conjecture and the Gaifman property. We discuss what we believe is known about the general statement of Gaifman's Conjecture, and compare full stability introduced here with the assumptions of Shelah's  \cite{Sh234}. On the one hand, Shelah's assumption of stability of systems is potentially significantly weaker than full stability. On the other hand, the assumption of ``no two cardinal model'' (that is assumed throughout \cite{Sh234}) is very strong (although it makes perfect sense in the context of Gaifman's Conjecture, since it follows easily from categoricity over $P$). In  section 9, we translate this assumption to the notion of nulldimensionality over $P$, which further exemplifies its strength. 

We should point out that the ongoing work of Shelah and the author (e.g., \cite{ShUs322a, ShUs322b}) is aimed at removing the assumption of no two-cardinal model, and establishing the Gaifman's property just under the  assumption of stability of $\mathcal P^-(n)$-systems. However,  this would not imply the results in the present paper, since here we establish \emph{full} existence (under the stronger assumption of full stability).

\subsection*{Acknowledgements} I am particularly grateful to  Zo\'{e} Chatzidakis. Her questions and our correspondence has motivated this line of research. I would also like to thank Itay Kaplan and Anass Alzurba for motivating and helpful discussions.

\section{The setting}

In this section we set up the framework and introduce some of the basic hypotheses and notions. The reader is referred to \cite{ShUs322a} for a more detailed discussion.

\subsection{Conventions and hypotheses}

\begin{conv}
Let $T$ be a complete first order theory, $P$ a monadic predicate in
its vocabulary. For simplicity, we will assume that the language of $P$ does not contain function symbols (so 
that every subset of a model, containing all the constant symbols, is a substructure; in fact, we will treat all subsets as substructures). In addition, we assume that $T$ implies that $P$ is infinite. 
\end{conv}

\noindent
Let $\mathcal{C} $ be the monster model of $T$. From now on, we assume that all models of $T$ are elementary submodels of $\mathcal{C}$, and all sets are subsets of $\mathcal{C}$. \medskip
 
For $M \models T$, we denote by  $M|_P$ the set $P^M$ viewed as a substructure of $M$. Similarly, for a subset $A \subseteq M$, we denote by $M|_A$ the substructure of $M$ with universe $A$. We write $A\equiv B$ if $Th({\cal C}|_A)=Th({\cal
C}|_B)$.

We also denote $\cC^P = \cC|_P$ and $T^P = Th(\cC|_P)$. For a set $A$, we denote $P^A = A \cap P^\cC$.

When no confusion should arise, we will write $P$ for $P^\cC$. Also, for a set $A$, we will often denote by $A$ both the set and the substructure of $\cC$ with universe $A$. So for example, when we write that $A\cap P^\cC$ is $\lam$-saturated, or just that $A\cap P$ is $\lam$-saturated, we mean that the substructure $\cC|_{A\cap P^\cC}$ is a $\lam$-saturated model of the appropriate theory (if $A \cap P \prec P$, which will be the case in this paper, then the appropriate theory is $T^P$).


%

\bigskip

Throughout the paper, we are going to make the following fundamental assumptions on $T$:

\begin{hyp}\label{asm:1} (\emph{\underline{Hypothesis 1})}

  
\begin{enumerate}
 \item Every type over $P^\mathcal{C}$ is definable. In other words, $P$ is \emph{stably embedded}: subsets of $P^\mathcal{C}$  that are definable  with parameters (in $\mathcal{C}$), are already definable (in $\mathcal{C}$) with parameters in $P^\mathcal{C}$.
\item  In addition, subsets of $P^\mathcal{C}$  that are 0-definable (in $\mathcal{C}$),  are already 0-definable in
$\cC^P$ (modulo $T^P$).
\item $T$ has quantifier elimination (even down to the level of predicates). 
 \end{enumerate}
\end{hyp}

Note that, combining (i) and (ii), one sees that the induced (from $\cC$) structure on $P^\cC$ is precisely the structure of $\cC^P$ (in the theory $T^P$). Indeed, if $\ph(\x,\c)$ defines (in $\cC$) a subset of $P^\cC$, then (by (i)) the same subset is defined by some formula $\theta(\x,\d)$ with $\d \in P$. At the same time, the subset defined in $\cC$ by $\theta(\x,\y)$ is also defined by some $\hat \theta(\x,\y)$ in $\cC^P$; so the original set is defined by $\hat \theta(\x,\d)$ in $\cC^P$. 

In addition, given (i) and (ii) above, since $T$ does not add any new definable sets to $\cC^P$, one can Morleyize $T$ (add a new predicate symbol for each formula in the language) without much cost; that is, Hypothesis 1(iii) above comes in a sense ``for free'' (see Observation 3.6 in \cite{ShUs322a} for more details). Conversely, clause (ii) actually follows immediately from clause (iii), so we could have omitted it. 
\smallskip

These assumptions are obviously very convenient, but there are also ``good'' reasons to assume them. First, they hold in many important and interesting examples. Second, it was shown by Pillay and Shelah in \cite{PiSh130} that these assumptions are in some sense ``justified'' in our context. Specifically, they have shown that, if either Hypotheses (i) or Hypothesis (ii) above fails, one can deduce  non-structure results ($\ge\lambda^+$ models of $T$ with the same $P$-part non-isomorphic over $P$ in every $\lam>|T|$ satisfying certain set-theoretic assumptions). Since we are currently mostly interested in theories $T$ that lie on the ``structure'' side of the classification theory dividing lines, we will  make these assumptions freely. 

At the same time, it is worth noting that, when considering specific examples, it is sometimes difficult to establish (or impossible to assume) clauses (ii) and/or (iii) of the Hypothesis above. At the end of section 7, we discuss this issue in the context of Lachlan's Theorem on full existence for countable stable theories. Specifically, we note that for some applications one may be able to Morleyize $T$, even if Hypothesis 1(ii) fails. For example, if we are interested (as we are here) in full existence, i.e., the existence property for all complete sets, we can work in the Morleyzation $T'$ (which will satisfy the assumptions above), and then deduce certain conclusions on the original $T$ (since the property of completeness is preserved between $T$ and $T'$). However, if one does not want to change $T^P$, the theory of $P$, its models and definable subsets (for example, if one is interested in Gaifman's Conjecture), this approach does not always work. 

\smallskip

Finally, note that $\cC^P$ can be seen as the monster model for the theory $T^P$, and that it follows from Hypothesis 1(iii) that $T^P$ also has QE. 


\subsection{The existence property, completeness, and stability}

The main concept under discussion in this paper is the existence property (see also Definition \ref{de:fullexist} below):

\begin{de}\label{de:existence}
\begin{enumerate}
\item 
	We say that a set $A$ has the \emph{existence property over $P$}, or simply the \emph{existence property} if there exists $M \models T$ such that $A \subseteq M$ and $P^M = P^A$. 
\item
	We say that $A$ has the \emph{full existence property over $P$}, or simply \emph{full existence} if every $B \subseteq A$ with $P^B = P^A$ has the existence property. 
\item
	Let $\lam$ be a cardinal. If, in addition, in clause (i) there is $M$ which is $\lam$-saturated, we say that $A$ has the \emph{$\lam$-existence property over $P$}, or simply the \emph{$\lam$-existence property}. Similarly for full $\lam$-existence.  
\item
	We say that $T$ has the \emph{Gaifman property} if every $N \models T^P$ has the existence property. That is, $T$ has the Gaifman property if every model of $T^P$ occurs as the $P$-part of some model of $T$. 
\end{enumerate}
\end{de}

Given a set $A$, in order for their to be a chance for $M$ as above to exist, $P^A$ should be ``suitable'' for being the $P$-part of a model of $T$. For example, $P^A$ should obviously be itself a model of $T^P$. However, this is not enough. $P^A$ should also be closed under the parameters necessary to define types of tuples of $A$ over $P$. These conditions are summarized in the following definition (see also Fact \ref{obs:complete_characterization}).


\begin{de}\label{6}\label{dfn:complete}
$A\subseteq {\cal C}$ is {\it complete} if for every
formula $\psi(\bar x,\bar y)$ and $\bar b\subseteq A, \models
(\exists\bar x\in P)\psi(\bar x,\bar b)$ implies $(\exists
\bar a\subseteq P\cap A)\models \psi(\bar a,\bar b)$. 
\end{de}

\begin{re}\label{re:lachlan} The notion of completeness was defined by Pillay and Shelah in \cite{PiSh130}, but it had been identified more than a decade earlier by Lachlan \cite{Lachlan1972} as a necessary condition for the existence property. Lachlan referred to this as ``the pair $(P,A)$ has the Tarski-Vaught property'', and proved that, if $T$ is countable and stable, the Tarski-Vaught property implies that $A$ has the existence property. See Corollary \ref{co:lachlan} below. 
 
\end{re}


The following is clear:

\begin{obs}\label{obs:complete}
If $M\prec {\cal C}$ and $P^M\subseteq A\subseteq M$, then
$A$ is complete.
\end{obs}

The last Observation confirms that in order for a set $A$ to have the existence property, it has to be complete. In this paper we are interested in the converse.



\medskip

The following useful characterization offers another understanding of the notion of completeness (see Observation 4.2 in \cite{ShUs322a}):

\begin{fact}
\label{obs:complete_characterization}
A set $A$ is complete if and only if for every $\bar a\subseteq A$ and
$\psi(\bar x,\bar y)$  the $\psi$-type $tp_\psi (\bar a/P^{\cal C})$ is definable
over $A\cap P^{\cal C}$ and $A\cap P^{\cal C}\prec P^{\cal C}$.
\end{fact}

Clearly, in order for $A$ to have the $\lam$-existence property, $P^A$ also has to be $\lam$-saturated.


%
%
%

\smallskip

As an easy consequence of Hypothesis \ref{asm:1}, we obtain that types over complete sets of elements ``in $P$'' are, 
in fact, types over $P^A$. Specifically (see Corollary 4.7 in \cite{ShUs322a}):

\begin{fact}\label{fct:typeoverP}
	Let $A$ be a complete set, $p(x)$ a (partial) type over $A$ with $P(x) \subseteq p$. Then $p$ is equivalent 
	to a $T^P$-type $p'$ over $P^A$ with $|p'|=|p|$. 
	
	In particular, if $p$ is finite, then it realized in $P^A$. Similarly, if $|p|<\lam$ and $P^A$ is $\lam$-compact.
\end{fact} 


\smallskip

It is also useful to note that, since $T$ has QE, the property of completeness for a set $A$  depends only on its first order theory (as a substructure of $\cC$):

\begin{lem}(Lemma 4.5 in \cite{ShUs322a})\label{lem:complete_QE}\label{7.5}
\begin{itemize}
\item[(i)] If $A_1\equiv A_2$, then $A_1$ is complete iff $A_2$ is
complete.
\item[(ii)] $A$ is complete iff whenever the
sentence $$\theta=:(\forall \bar y)[S(\bar y)\longleftrightarrow (\exists x\in P)
R(x,\bar y)]$$ for quantifier free 
$R,S$ is satisfied in ${\cal C}$, then
$A$ satisfies $\theta$.
%
%
\end{itemize}
\end{lem}

%

\smallskip

Let us now recall the relevant notion of type for this context. Note that it is only defined over a complete set.

\begin{de}\label{6.5}\label{dfn:startypes}

Let $A$ be a complete set. 

\begin{itemize}

\item[(i)] Let
$$S_*(A)=\{tp(\bar c/A):P\cap (A\cup \bar c)=P\cap A {\rm\ and}\ A\cup \bar
c\ {\rm is\ complete}\}$$
\item[(ii)] $A$ is \emph{stable over $P$}, or simply \emph{stable}, if for all $A^\prime$
with $A^\prime\equiv A$, we have $|S_*(A^\prime)|\leq |A^\prime|^{|T|}$.
\end{itemize}
\end{de}

\begin{remark}
\begin{enumerate}
\item Even though ``stability over $P$'' is a more appropriate and accurate name for our notion of stability of a set (and the term ``stable set'' exists in literature, and has a different meaning), since we  have only one notion of stability in this article (stability over $P$), we will sometimes omit ``over P'' and simply write ``stable''. 
\item Sometimes we refer to types in  $S_*(A)$ as \emph{complete types over $A$ which are weakly orthogonal to $P$}. 
\end{enumerate}

\end{remark}

%

\begin{obs}\label{8.5}
 Let $A$ be a set and $\c$ a tuple. Then $A$ is complete and $tp(\c/A) \in S_*(A)$ if and only if for every formula $\psi(\x,\b,\z)$ over $A$ we have 
 \[ \models \exists \z \in P\, \psi(\c,\b,\z) \Longrightarrow \exists \bar d \in P\cap  A \, \text{such that} \, \models \psi(\bar c,\b,\d) \]
 \end{obs}

\medskip

Now that we have defined complete sets, we can conclude this session with the  definition of the main property that we are after in this paper.

\begin{de}\label{de:fullexist}
 \begin{enumerate}
\item We say that a theory $T$ has the \emph{full existence property}, or just has \emph{full existence}, if every complete set $A$ has the existence property. 
\item Given $N \models T^P$, we say that $T$ has  \emph{full existence over $N$} if every complete set $A$ with $P^A = N$ has the existence property. 
\end{enumerate}

\end{de}

Let us recall one known general example of the full existence property:

\begin{fact}(\cite[4.13]{ShUs322a})\label{fact:satexist}
Let $N$ be a saturated model of $T^P$. Then $T$ has the full existence property over $N$.
\end{fact}

\section{Stability and rank}

Next, let us recall a notion of rank that ``captures'' stability over $P$ (\cite{PiSh130}, \cite{ShUs322a}). It will play a central role in our arguments. 

\smallskip

\begin{de}\label{R}
For a complete set $A$, a (partial) $n$-type $p(\bar x)$ (with parameters in ${\cal C}$), 
  sets $\Delta _1,\Delta _2$ of formulas $\psi (\bar
x,\bar y)$, and a cardinal $\lambda$, we define when $R^n_A(p,\Delta
_1,\Delta _2,\lambda)\geq \alpha$. We usually omit $n$.

\begin{itemize}
\item[(i)] $R_A(p,\Delta _1,\Delta _2,\lambda)\geq 0$ if $p(\bar x)$ is consistent. 

\item[(ii)] For $\alpha$ a limit ordinal: $R_A(p,\Delta _1,\Delta
_2,\lambda)\geq \alpha$ if $R_A(p,\Delta _1,\Delta
_2,\lambda)\geq \beta$ for every $\beta <\alpha$.

\item[(iii)] For $\alpha =\beta +1$ and $\beta$ even:
For $\mu<\lambda$ and finite $q(\bar x)\subseteq p(\bar x)$ we can
find $r_i(\bar x)$ for $i\leq\mu$ such that;

\begin{itemize}
\item[{1.}] Each $r_i$ is a $\Delta _1$-type over $A$,

\item[{2.}] For $i\not= j, r_i$ and $r_j$ are explicitly contradictory
(i.e. for some $\psi$ and $\bar c$, $\psi (\bar x,\bar c)\in r_i,
\neg\psi(\bar x,\bar c)\in r_j$).

\item[{3.}] $R_A(q(\bar x)\cup r_i(\bar x), \Delta _1,\Delta
_2,\lambda)\geq \beta$ {for all $i$}.
\end{itemize}

\item[(iv)] For $\alpha =\beta +1: \beta$ odd: For 
$\mu<\lambda$ and finite $q(\bar x)\subseteq p(\bar x)$ and $\psi
_i\in \Delta _2, \bar b_i\in A$ ($i\leq \mu$), there are $\bar d_i\in
A\cap P$ such that $R(r_i,\Delta _1,\Delta
_2,\lambda)\geq \beta$ where $r_i=q(\bar x)\cup \{(\forall\bar
z\subseteq P) \left[\psi _i(\bar x,\bar b_i,\bar z)\equiv\Psi_{\psi
_i}(\bar z,\bar d_i)\right] \colon i<\mu\}$, and $\Psi_{\psi _i}$ is as in Observation \ref{obs:defcomplete}.
\end{itemize}

$R^n_A(p,\Delta _1,\Delta _2,\lambda)= \alpha$ if $R^n_A(p,\Delta
_1,\Delta _2,\lambda)\geq \alpha$ but not $R^n_A(p,\Delta
_1,\Delta _2,\lambda)\geq \alpha +1$. $R^n_A(p,\Delta
_1,\Delta _2,\lambda)=\infty$ iff $R^n_A(p,\Delta
_1,\Delta _2,\lambda)\geq \alpha$ for all $\alpha$.
\end{de}

The main case for applications will be $\lambda =2$. Note that the
larger $R^n_A(p,\Delta _1,\Delta _2,\lambda)$, the more evidence there
is for the existence of many types $q(\bar x)\in S_*(A)$ consistent with
$p(\bar x)$.

\medskip 

See section 5 of \cite{ShUs322a} for a detailed discussion of some basic properties of the rank; we recall here only a few. 
\medskip

\begin{fact}\label{rank}(Fact 5.3 in \cite{ShUs322a})
%
For every $p$ there is a finite $q\subseteq p$, such that 
$R_A(p,\Delta _1,\Delta _2, 2)=R_A(q,\Delta _1,\Delta _2, 2)$.
%
%
\end{fact}

\begin{fact}\label{rankeven}(see Fact 5.3 in \cite{ShUs322a}). 
	Let $A$ be complete, $p\in S_*(A)$, $q^* \subseteq p$, and assume
	 $$R^n_A(q^*,\Delta _1,\Delta _2,\lambda)=R^n_A(p,\Delta _1,\Delta _2,\lambda)=k < \infty$$ Then $k$ is even.
\end{fact}

The following Theorem allows us to make use of the rank under the assumption of stability. 

\begin{theorem}\label{13}\label{thm:stablerank}(\cite{PiSh130}, see also Theorem 5.4 in \cite{ShUs322a})

The following are equivalent:
\begin{itemize}
\item[(i)] $A$ is stable.
\item[(ii)] For every finite $\Delta _1$ and finite $n$ there are some
finite $\Delta _2$ and finite $m$ such that 
$R^n_A(\bar x=\bar x,\Delta _1,\Delta _2,2)\leq m$.
\end{itemize}
\end{theorem}

\begin{fact} (Corollary 5.5 in \cite{ShUs322a})\label{fct:satstable}
	In Definition \ref{6}(iv), it is not necessary to consider all $A' \equiv A$. 
	Specifically, a complete set $A$ is stable if and only if $|S_*(A')| \le |A'|^{|T|}$ for some $A' \equiv A$ saturated, $|A'|>|T|$. 
	
	Moreover, it is enough that for $A'$ as above, $|S_*(A')| < 2^{|A'|}$.
\end{fact}

\medskip

We  often omit the superscript and the subscript in the rank $R^n_A$, and write simply $R$ (at least when 
$n$ and $A$ are easily deduced from the context).

\medskip

It  follows (see  Corollary 5.6 in \cite{ShUs322a}) that every type $p \in S_*(A)$ over a stable set $A$ is definable internally in $A$:



\begin{co}\label{14} 

\begin{enumerate}
\item
If $A$ is stable, then for every $\psi (\bar
x,\bar y)\in L(T)$ there is $\Psi _\psi$ in $L(A)$ such that if $p\in
S_*(A)$, then for some $\bar b\subseteq A, \Psi _\psi (\bar
y,\bar b)$ defines $p| \psi$ in $\cC_A$. 

That is, for every $\c \in A$, $\psi(\x,\c) \in p$ if and only if $A \models \Psi _\psi (\bar
c,\bar b)$.
 \item Moreover, if $|A|\geq 2$, then for every $\psi(\x,\y)$, there is a definition $\Psi_\psi(\x,\y)$ as above which works uniformly for all $B \equiv A$ and $p \in S_*(B)$. 
 \end{enumerate}
\end{co}

\section{More on completeness}

In this section we take a closer look at the notion of completeness, and prove several easy ``internal'' characterizations that will become very useful later. 

\begin{obs}\label{obs:uniformdef}
In Hypothesis \ref{asm:1} (i), the definitions of $\psi$-types over $P$ can be assumed to be \emph{uniform}. Specifically:

 There are $\langle \Psi_\psi(\y,\z) :\psi(\bar
x,\bar y)\in L(T)\rangle $ such that for all
$\bar a\subseteq \cC$, $tp_\psi (\bar a/P^\cC)$ is definable by
$\Psi _\psi (\bar y,\bar c)$ for some $\bar c\subseteq P^\cC$. 

In other words, for every $\psi(\x,\y)$ and $\a \in \fC$ there exists $\c \in P^\cC$ such that $\cC \models \forall \y \psi(\a,\y) \longleftrightarrow \Psi(\y,\c)$.

\end{obs}
\begin{proof}
 This is a standard  compactness argument. First, note that by compactness, for every $\psi(\x,\y)$ there is a finite number of possibilities for $\Psi_\psi(\y,\z)$. Otherwise, by considering the following:
 
 \[
\left\{ 
 \forall \z  \neg\forall \y \left(\psi(\x,\y) \longleftrightarrow  \Psi(\y,\z)\right) \colon \Psi(\y,\z)
\right\} 
 \]
 
one constructs an undefinable $\psi$-type.
 
Since $|P^\cC| \ge 2$, one can now easily manipulate these into a single defining formula. 
\end{proof}

\bigskip

\begin{obs} \label{obs:defcomplete}For any complete $A$ and
for all
$\bar a\subseteq A$, $tp_\psi (\bar a/P\cap A)$ is definable by
$\Psi _\psi (\bar y,\bar c)$ for some $\bar c\subseteq A\cap P$
(where $\Psi _\psi (\bar y,\bar z)$ is as in Observation \ref{obs:uniformdef}).
\end{obs}
\begin{proof}
 Let $\bar a\subseteq A$. We know (by the choice of  $\Psi _\psi (\bar y,\bar z)$) that 
 
 \[
 	\cC \models (\exists \z \in P)\left( \forall \y \psi(\a,\y) \longleftrightarrow \Psi _\psi (\bar y,\bar z)\right)
 \]
 Since $A$ is complete, there is already such a $\c \in P^A$. 
 
\end{proof}

In \cite{ShUs322a} we stated this with dependence on $A$, but there is clearly no need.

\smallskip

This gives us the following useful characterization of complete sets, slightly strengthening Fact \ref{obs:complete_characterization}:

\begin{co}
 \label{co:defcomplete} Let $A$ be a set. The following are equivalent:
\begin{enumerate}
\item $A$ is complete
\item  $P^A \elem P^\cC$, and for any formula  $\psi(\a,\y)$ over $A$ (i.e., $\a \in A$) there is $\d \in P^A$ such that:
\[ \cC\models \forall \y \left[\psi(\a,\y) \longleftrightarrow \Psi _\psi (\bar y,\bar d)\right] \]
\end{enumerate}

\end{co}

A similar characterization can be formulated on the level of types:

\begin{co}\label{co:stardefincluded}
 Let $A$ be a complete set and $\a$ a tuple. Then $tp(\a/A) \in S_*(A)$ if and only if for every formula $\psi(\x,\b,\y)$ over $A$ there is $\d \in P^A$ such that 
\[
\models (\forall \y \subseteq P) \left[ \psi(\a,\b,\y) \longleftrightarrow \Psi_\psi(\y,\d) \right] 
\]
where  $\Psi_\psi(\y,\z)$ is the defining formula as in Observation \ref{obs:defcomplete} for the formula $\psi = \psi(\x\x',\y)$; so $\x\x'$ are the variables of $\psi$, and $\y$ are the parameters, where $\len(\x') = \len(\b)$.
 \end{co}

%
%

\begin{re}\label{re:stardefinincluded}
 A different way of phrasing the previous Corollary is: 
 
  Let $A$ be a complete set and $p \in S(A)$. Then $p \in S_*(A)$ if and only if for every formula $\psi(\x,\b,\y)$ over $A$ there is $\d \in P^A$ such that 
\[
(\forall \y \subseteq P) \left[ \psi(\x,\b,\z) \longleftrightarrow \Psi_\psi(\y,\d) \right]  \in p
\]
Where $\Psi_{\psi}(\y,\z)$ is as in the previous Corollary. 
\end{re}

\medskip

Combining some of the observations above (Observation \ref{8.5} and Corollary \ref{co:stardefincluded}), we can conclude a convenient ``internal'' characterization of being a *-type over a complete set:

\begin{co}\label{co:type}
 Let $A$ be a  set and $p = p(\x)  \in S(A)$. Then the following are equivalent:
 \begin{enumerate}
\item $A$ is complete and $p \in S_*(A)$
\item For every formula $\psi(\x,\b,\z)$ over $A$ there exists $\d \in P^A$ such that 

\[
	\left[ \left((\exists \z \subseteq P)  \psi(\x,\b,\z)\right) \longrightarrow \psi(\x,\b,\d) \right] \in p
\]
\item
$P^A \elem P^\cC$, and for every formula $\psi(\x,\b,\y)$ over $A$ there is $\d \in P^A$ such that 
\[
\left[(\forall \y \subseteq P) \left( \psi(\x,\b,\y) \longleftrightarrow \Psi_\psi(\y,\d) \right) \right]  \in p
\]
where  $\Psi_\psi(\y,\z)$ is as in Corollary \ref{co:stardefincluded}.
\end{enumerate}

\end{co}

\medskip 

\noindent
Let us extend these ideas to a slightly less trivial characterization.

\smallskip
\begin{theorem}\label{co:complete_char}
Let $A$ be a set. Then the following are equivalent:
\begin{enumerate}
\item $A$ is complete
\item For every finite type $p_0(\x)$ over $A$ (not necessarily realized in $A$) and every finite collection \set{\psi_i(\x,\b_i,\z_i) \colon i<k} of formulae over $A$, there are  $\set{\d_i \colon i<k} \subseteq P^A$ such that the following is a (finite) type over $A$:

\[
	p_0(\x) \cup \left\{  \left[(\exists \z_i \subseteq P)  \psi_i(\x,\b_i,\z_i)\right] \longrightarrow \psi_i(\x,\b_i,\d_i)  \colon 
	i<k \right\}
\]

\item For every finite type $p_0(\x)$ over $A$ (not necessarily realized in $A$) and every formula $\psi(\x,\b,\z)$ over $A$, there is  $\d \subseteq P^A$ such that the following is a (finite) type over $A$:

\[
	p_0(\x) \cup \left\{  \left[(\exists \z \subseteq P)  \psi(\x,\b,\z)\right] \longrightarrow \psi(\x,\b,\d)   \right\}
\]

\item
$P^A \elem P^\cC$, and:

For every finite type $p_0(\x)$ over $A$ (not necessarily realized in $A$) and every finite collection \set{\psi_i(\x,\b_i,\y_i) \colon i<k} of formulae over $A$ , there are  $\set{\d_i \colon i<k} \subseteq P^A$ such that  the following is a (finite) type over $A$:
\[
p_0(\x) \cup \left\{  (\forall \y_i \subseteq P) \left[ \psi_i(\x,\b_i,\y_i) \longleftrightarrow \Psi_{\psi_i}(\y_i,\d_i) \right]  \colon 
	i<k \right\}
%
\]
where  $\Psi_{\psi_i}(\y_i,\z_i)$ is the defining formula for $\psi_i = \psi_i(\x\x^i, \y_i)$ as in Corollary \ref{co:stardefincluded}; so $\len(\x^i) = \len(\b^i)$.

\item
$P^A \elem P^\cC$, and:

For every finite type $p_0(\x)$ over $A$ (not necessarily realized in $A$) and a formula $\psi(\x,\b,\y)$  over $A$ , there is  $\d \subseteq P^A$ such that  the following is a (finite) type over $A$:
\[
p_0(\x) \cup \left\{  (\forall \y \subseteq P) \left[ \psi(\x,\b,\y) \longleftrightarrow \Psi_{\psi}(\y,\d) \right]  \colon 
	i<k \right\}
%
\]
where  $\Psi_\psi(\y,\z)$ is the defining formula for $\psi = \psi(\x\x',\y)$.

\end{enumerate}

\end{theorem}
\begin{proof}
 (i) $\implies$ (ii): Assume $A$ is complete, and let $p_0$, $\psi_i$ be as in (ii). Consider the following formula:
 
 \[
 	\Phi(\z_0, \ldots, \z_{k-1}) = \exists \x \left(\bigwedge p_0(x) \bigwedge_{i<k} \left[(\exists \z_i \subseteq P)  \psi_i(\x,\b_i,\z_i)\right] \longrightarrow \psi(\x,\b_i,\z_i)\right)
 \]
 
 Clearly, $\cC \models (\exists \z\subseteq P) \Phi(\z)$ (any realization of $p$ can be chosen as $\x$).  Since $A$ is complete (and $\Phi$ is over $A$), there exist $\d = \seq{d_i \colon i<k} \subseteq P^A$ such that $\cC \models \Phi(\d)$; it is easy to see that they are as required in (ii). 
 
 (ii) $\implies$ (iii) and (iii) $\implies$ (i) are clear (for the latter, take e.g. $p_0 = [x=x]$). 

Similarly for (iv) $\implies$ (v) and (v) $\implies$ (i).
 
 
%

 (i) $\implies$ (iv): Since $A$ is complete, $P^A \elem P^\cC$.  Now let $p_0$, $\psi_i$ be as in (iv). Similarly to the previous proof, consider 
 
\[
 	\Phi(\z_0, \ldots, \z_{k-1}) = \exists \x \left(\bigwedge p_0(x) \bigwedge_{i<k} (\forall \y_i \subseteq P) \left[ \psi_i(\x,\b_i,\y_i) \longleftrightarrow \Psi_{\psi_i}(\y_i,\z_i) \right] \right)
 \]
 
Again, we claim that $\cC\models (\exists \z \subseteq P) \Phi( \z)$.  Indeed, we are not asking (yet) that $\z$ be in $P^A$; so taking any $\a \models p$ (for $\x$), and any $\d'_i \in P^\cC$ as in Observation \ref{obs:uniformdef} for each $\psi_i(\a\b_i,\y_i)$ works. But now, since $\Phi(\z)$ is over $A$, and $A$ is complete, we can find $\d$ as required in $P^A$.

\end{proof}

\medskip

Our next goal is to investigate the notion of completeness under the assumption of elimination of imaginaries. Recall that a tuple $c \in \cC$ is called a code (or a canonical parameter) for a definable set $X$ if there is a formula $\ph(x,y)$ such that $c$ is the unique tuple satisfying $X = \ph(x,c)$. One can characterize codes in terms of global automorphisms. Specifically, $c$ is a code for $X$ if and only if: for every every automorphism $\sigma$ of $\cC$, $\sigma$ fixes $X$ (setwise) if and only if it fixes $c$ pointwise. 

Recall that $T$ is said to eliminate imaginaries if every 0-definable set has a code (which implies that every definable set has a code). 

\begin{lem}\label{lem:eq}
 Assume that $T^P$ eliminates imaginaries. Then every definable subset of $P$ has a code in $P$. 
\end{lem}
\begin{proof}
 This is a straightforward consequence of Hypothesis 1 (Hypothesis \ref{asm:1}). 
 
 Indeed, let $\ph(x,a)$ define a subset  $X$ of $P$. By Hypothesis 1, there is a formula $\theta(x,c)$ (with $c \in P^\cC$) that defines $X$ in $\cC^P$. Let $c_E$ be a code for $X$ in $\cC^P$. Then for every automorphism $\sigma$ of $\cC^P$, $\sigma$ fixes $X$ (as a set) if and only if it fixes $c_E$. Clearly this is also true then for any automorphism of $\cC$. Hence $c_E$ is also a code for $\ph(x,a)$ in $\cC$. 
\end{proof}

\begin{theorem}\label{thm:eq}
 Assume that $T$, or just $T^P$, eliminates imaginaries. Then the following are equivalent for a set $A$:
 \begin{enumerate}
\item $A$ is complete
\item $P^A \elem P^\cC$, and $\dcl(A) \cap P = P^A$
\item $P^A \elem P^\cC$, and $\acl(A) \cap P = P^A$
\end{enumerate}

\end{theorem}
\begin{proof}
 \noindent (i) $\implies$ (iii) [Does not require elimination of imaginaries].

Let $b \in acl(A) \cap P$, and suppose that $b = b_0, \ldots, b_{k-1}$ are all the solutions to the algebraic formula $\ph(x,a)$ over $A$. Since $A$ is complete, the set $B = \set{b_i\colon i<k}$ is definable over $P^A$, and, therefore, by Hypothesis \ref{asm:1}(ii), it is definable in $P^A$ (in $T^P$). Hence $b \in \acl^{T^P}(P^A) = P^A$.  

 \noindent (iii) $\implies$ (ii) is obvious. 
 
  \noindent (ii) $\implies$ (i).
 Assume (ii); we need only prove that for every $a \in A$, the type $\tp(a/P)$ is definable over $P^A$. This follows from the fact that the canonical parameters (in $\cC^P$) of the defining formulas are all definable over $A$ (in $\cC$). For the sake of completeness, we sketch the argument. 
 
 Let $\ph(x,y)$ be a formula, and let $c_E$ be the code for the set $X = P \cap \ph(a,y)$ as in Lemma \ref{lem:eq}. Clearly, $X$ is definable over $c_E$, and $c_E$ is definable over $a$ (since every automorphism of $\cC$ that fixes $a$ will also fix $X$, hence $c_E$). So $c_E \in \dcl(A) \cap P \subseteq P^A$, as required.  
 
%
%
 
\end{proof}

\begin{co}\label{co:eq}
 	Assume $T^P$ eliminates imaginaries. Let $A$ be a set with $P^A \elem P^\cC$, and $a \in \cC$ a tuple. Then $A$ is complete and $\tp(a/A) \in S_*(A)$ if and only if $\dcl(Aa \cap P) = P^A$ if and only if $\acl(Aa \cap P) = P^A$.
\end{co}

\section{Minimal types}

One of the most basic tools in stability theory is the existence of minimal types. More generally, whenever a theory admits a ``reasonable'' notion of rank for partial types, every type (with a rank) with a particular property can be extended to a minimal such type. This is also true in our context:

\begin{lem}\label{lem:minimal}
	Let $A$ be complete,
$p_0$ a 
partial type over $A$, and let $X$ be a property (not necessarily definable) of partial types over $A$ such that $X$ holds for $p$. Assume, in addition, that $R^n_A(p,\Delta _1,\Delta _2,\lambda)=k < \omega$. 

Then there exists $p_1$ such that:

\begin{enumerate}
 \item
 	$p_1$ is a partial type over $A$
\item
	$p_0 \subseteq p_1$
\item 
	$X$ holds for $p_1$
\item
	$R^n_A(p_1,\Delta _1,\Delta _2,\lambda)$ is minimal with respect to the previous clauses. 
\end{enumerate}
\end{lem}

We call a type $p'$ satisfying the clauses (i) -- (iv) of the Lemma above \emph{a minimal extension of $p$ with respect to $\Delta_1$, $\Delta_2$, $\lam$, and $X$}. If $p$ is clear from the context, or $p = [x=x]$, we say that $p'$ is a \emph{minimal type} with respect to $\Delta_1$, $\Delta_2$, $\lam$, and $X$. If $\Delta_1$, $\Delta_2$, $\lambda$ are also clear from the context (e.g., in this paper $\lam$ will always be 2), then we simply say that $p'$ is minimal with respect to $X$. 

\medskip

However, one has to be careful with what it means for a type to be minimal in our case. In classical stability theory,  minimality with respect to Shelah's  2-rank (and property $X$) has several strong consequences, such as uniqueness of extensions with the prescribed property. This would be the case in our setting with a minimal type of even rank. Indeed, given a type $p$, which is minimal with respect to  $\Delta_1$, $\Delta_2$, $\lam=2$, $X$, whose rank is even, by the definition of the rank, one can not split $p$ into two two contradictory types, still satisfying the property $X$, using a $\Delta_1$-formula, which leads to uniqueness, at least with respect to $\Delta_1$-extensions. By iterating over all formulas, one obtains a type which is minimal with respect to $X$. However, the problem is that we can not rule out the case of the rank $k$ being odd. The following Lemma shows that, at least in certain cases, we do not need to worry about this. This will allow us to obtain ``truly'' minimal types in these situations. 

%

\begin{lem}\label{lem:rankeven}
 	
	Let $A$ be complete,
	$p$ a finite partial type over $A$ such that $R^n_A(p,\Delta _1,\Delta _2, \lam) < \omega$ for some $\Delta_1, \Delta_2, \lambda$. 
	
	Let 
	$p'$ be a minimal extension of $p$ with respect to $\Delta_1$,   $\Delta_2$, $\lam$, and the property ``finite''. Let $R^n_A(p',\Delta _1,\Delta _2,\lambda)=k < \omega$. Then $k$ is even.


\end{lem}

\begin{proof}
Assume $k$ is odd; we shall show that 
$R^n_A(p',\Delta _1,\Delta _2,\lambda)\ge k+1$. 
	
\smallskip

Let $\psi
_i\in \Delta _2, \bar b_i\in A$ as in clause (iv) of the definition of the rank (Definition \ref{R}).  By Corollary \ref{co:complete_char}, since $A$ is complete,  
there are $\d_i \in P^A$ such that $p'(\x)$ is consistent with the set 
\[
	\pi(\x,\d) = 
	\left\{(\forall\bar z_i \subseteq P) \left[\psi _i(\bar x,\bar d_i,\bar z_i)\longleftrightarrow \Psi_{\psi_i}(\bar z_i,\bar c_i)\right] \colon i
  \right\}
\]

Let $p''(\x) = p'(\x) \cup \pi(\x,\c)$. It is still a finite type over $A$ 
Hence by minimality of $p'$, $R(p''(\x), \Delta_1, \Delta_2, \lam) = R(p'(\x), \Delta_1, \Delta_2, \lam) = k$. 

Now let $r_i=q(\bar x)\cup \{(\forall\bar
z\subseteq P) \left[\psi _i(\bar x,\bar b_i,\bar z)\equiv\Psi_{\psi
_i}(\bar z,\bar d_i)\right]$ for all  $i$ (again, as in clause (iv) of the definition of the rank). Note that $r_i \subseteq p''(\x)$, hence 
$R(r_i(\x), \Delta_1, \Delta_2, 2) \ge R(p''(\x), \Delta_1, \Delta_2, 2) \ge k$ for all $i$. 

So by (iv) of Definition \ref{R}, $R_A(p,\Delta _1,\Delta _2,\lam)\geq k+1$, and we are done.


\end{proof}

\medskip

\section{Relevant notions of isolation}

%

Let us recall the definitions of the notions of isolation relevant for the discussion in this article. 

\begin{de}
\begin{enumerate}
\item A (partial) type $p$ over a set $A$ is called \emph{locally isolated} (l.i.) if for every formula $\ph(x,y)$ there exists a formula $\theta_\ph(x,a_\ph) \in p$ such that $\theta_\ph(x,a_\ph) \vdash p\rest\ph$. We say that $p$ is locally isolated (l.i.) over $B \subseteq A$ if for every $\ph(x,y)$, its isolating formula $\theta_\ph$ is over $B$ (so $a_\ph \in B$). 

\item A (partial) type $p$ over a set $A$ is called \emph{$\lam$-isolated} if there exists a subset $r \subseteq p$ with $|r|<\lam$ such that $r \equiv p$. We say that $p$ is $\lam$-isolated over $B \subseteq A$ if $r$ as above is a partial type over $B$.

\end{enumerate}

\begin{re}
\begin{enumerate}
\item So $p$ is locally isolated if for every formula $\ph$, the restriction of $p$ to a $\ph$-type is implied by a single formula in $p$ (which does not itself have to be a $\ph$-formula). 
\item A type $p$ is isolated iff it is $\aleph_0$-isolated. 
\end{enumerate}
 
\end{re}

\smallskip

The following Lemma with play a crucial role in our constructions It states that a locally isolated type over a complete set is \emph{always} weakly orthogonal to $P$.

\begin{lem}\label{lem:isolated_star}
 Let $B$ be a complete set, $p \in S(B)$ locally isolated. Then $p \in S_*(B)$. 
\end{lem}
\begin{proof}
Consider the formula $\theta(\x,\y,\bar u) = (\forall \z \subseteq P)\left[ \psi(\x,\y,\z) \longleftrightarrow \Psi_\psi(\z,\y,\bar u)\right]$. 

 By Remark \ref{re:stardefinincluded}, it is enough to show that for every formula $\psi(\x,\b,\z)$ over $A$ there is $\d \in P^A$ such that 
\[
\theta(\x,\b,\d) = (\forall \y \subseteq P) \left[ \psi(\x,\b,\y) \longleftrightarrow \Psi_\psi(\y,\d) \right]  \in p
\]
where  $\Psi_\psi(\y,\z)$ is as in Corollary \ref{co:stardefincluded}.

Since $p$ is locally isolated, there is a finite $p_0 \subseteq p$ such that $p_0 \vdash p \rest \theta$. 

By clause (v) of Theorem \ref{co:complete_char}, 
 there is  $\d \subseteq P^A$ such that  the following is a type over $A$:
\[
\pi(\x) = p_0(\x) \cup \left\{  \theta(\x,\b,\d) 
\right\}
%
\]

In other words, there is a complete type $q$ over $B$ extending $\pi(\x)$. But $p_0$ implies a complete $\theta$-type; so $q\rest\theta = p\rest\theta$. In particular, $\theta(\x,\b,\d) \in p$, as required.

\end{proof}

\end{de} 

\begin{de}\label{dfn:lprimary}
Let $N$ be a model, $P^N \subseteq B \subseteq N$.
\begin{enumerate}
\item
	We say that the sequence $\d = \{d_i:i<\alpha\} \subseteq N$ is a \emph{local construction} over $B$ in $N$ if for all $i<\al$, the type $tp\{d_i/B\cup\{d_j:j<i\})$ is locally isolated.

\item
	We say that a set $C \subseteq N$ is $N$ is \emph{ locally constructible} (l.c.) over $B$ in $N$ if 
	there is a local construction $\d$ over $B$ in $N$. 
	
\end{enumerate}
\end{de}

\begin{de}\label{dfn:primary}
Let $N$ be a model, $P^N \subseteq B \subseteq N$.
\begin{enumerate}
\item
	We say that the sequence $\d = \{d_i:i<\alpha\} \subseteq N$ is a $\lam$-construction over $B$ in $N$ if for all $i<\al$, the type $tp\{d_i/B\cup\{d_j:j<i\})$ is $\lambda$-isolated.

\item
	We say that a set $C \subseteq N$ is $N$ is \emph{$\lam$-constructible} over $B$ in $N$ if 
	there is a $\lam$-construction $\d$ over $B$ in $N$. 
	
	In particular, we say that $N$ is
	\emph{$\lam$-constructible} over $B$ if 
there is a construction $N=B\cup\{d_i:i<\al\}$ such that for all $i<\lam$ the type 
$tp\{d_i/B\cup\{d_j:j<i\})$ is $\lambda$-isolated. 
\item 
	We say that a model $N$ is \emph{$\lambda$-primary} over $B$ if it is $\lam$-constructible and $\lam$-saturated. 
\end{enumerate}
\end{de}

\begin{de}\label{dfn:prime}
Let $N$ be a model, $P^N \subseteq B \subseteq N$.
\begin{enumerate}
\item 

	We say that $N$ is \emph{locally atomic} (l.a.) over $B$ if 
for every $\bar d\subseteq N$, $tp(\bar d,B)$ is
locally isolated over $B$. 

\item 

	We say that $N$ is \emph{$\lam$-atomic} over $B$ if 
for every $\bar d\subseteq N$, $tp(\bar d,B)$ is
$\lambda$-isolated over some $B_{\bar d}\subseteq B, |B_{\bar
d}|<\lambda$.
%
%
\end{enumerate}
\end{de}

\begin{obs}\label{obs:localimpliestplus}
	Let $A$ be a set, and $p\in S(A)$ locally isolated. Then $p$ is $|T|^+$-isolated. 
\end{obs}


\section{The existence property}


In this section we prove the full existence property for fully stable countable theories. 


\begin{lem}\label{lem:li} 
Assume that $T$ is countable. 
\begin{enumerate}
\item 
Let $B$ be a stable set, $p(\bar x)$  be
a finite $m$-type over $B$, $\psi(\x,\y)$ a formula. Then there is
a finite $\psi$-type $q(\bar x)$ over $B$ 
such that 
$p(\bar x)\cup q(\bar x)$ is consistent, and it implies a complete $\psi$-type over $B$. 


\item 
Let $B$ be a stable set, $p(\bar x)$  be
a finite $m$-type over $B$. Then there is
$q(\bar x)$ such that $|q(\bar x)|\leq |T| = \aleph_0$,
$p(\bar x)\cup q(\bar x)$ consistent, and there is $r\in S_*(B)$ such
that $r$ is locally isolated, and $p(\bar x)\cup q(\bar x)\equiv r(\bar x)$.

In particular,  $r(\bar x)$ is
$\aleph_1$-isolated.

\end{enumerate}

\end{lem}

\begin{proof} 

\begin{enumerate}
\item
Let $\Delta $ be finite such that $R(\bar x=\bar
x,\{\psi\},\Delta,2)<\omega$ (where $R = R^m_B$). Define $q(\bar x)$ 
such that
\begin{itemize}
\item[(a)] $q$ is finite and is over $B$,
\item[(b)] $p(\bar x) \cup q(\bar x)$ is
consistent, and
\item[(c)] $R(p\cup q,\{\psi \},\Delta,2)$ is
minimal with respect to (a) and (b).
\end{itemize}

This is possible since $B$ is stable. Moreover, by Lemma \ref{lem:rankeven}, $R(p\cup q,\{\psi\},\Delta ,2)$ is even. 

In particular, by the definition of the rank, we have: 

\begin{itemize}
\item[(d)]
For no $\bar b\subseteq B$ do we have $R(p\cup q\cup \{\pm\psi(\bar x,\bar b)\},\{\psi \},\Delta
,2)\geq R(p\cup q,\{\psi \},\Delta ,2)$
\end{itemize}


We claim that $p \cup q$ implies a complete $\psi$-type over $B$. Assume the contrary; so there exists $\b \in B$ such that 
$p \cup q \cup \{\pm \psi(\x,\b)\}$ are both consistent. By (d) above, one of the ranks $R(p\cup q\cup \{\pm\psi(\bar x,\bar b)\},\{\psi \},\Delta
,2)$ has to be (strictly) smaller than $R(p\cup q,\{\psi \},\Delta ,2)$. Without loss of generality, $R(p\cup q\cup \{\neg\psi(\bar x,\bar b)\},\{\psi \},\Delta
,2)<  R(p\cup q,\{\psi \},\Delta ,2)$. However, $q \cup \{\neg\psi(\bar x,\bar b)\}$ satisfies clauses (a), (b) above; so we get a contradiction to minimality (clause (c)).


\item 
Let $\{\ps _i(\bar x,\bar y_i):i<\omega\}$ list all
formulas of $L(T)$. Let $\Delta _i$ be finite such that $R(\bar x=\bar
x,\{\psi _i\},\Delta _i,2)<\omega$ (where $R = R^m_B$, as before). 

Define $q_i(\bar x)$ by induction
on $i < \omega$ such that
\begin{itemize}
\item[(a)] $q_i$ is finite and is over $B$,
\item[(b)] $p(\bar x)\cup \bigcup _{j\leq i}q_j(\bar x)$ is
consistent, and
\item[(c)] $p\cup\bigcup _{j\leq i}q_j$ implies a complete $\psi_i$-type over $B$.
%
\end{itemize}

This is possible by clause (i) (note that at every stage, the type $p\cup\bigcup _{j\leq i}q_j$ is finite). 


Let $q = \bigcup _{j\leq \aleph_0}q_j$. Clearly $|q| \le \aleph_0$. By clause (c) of the construction, $q$ implies a complete type over $B$. Let $r \in S(B)$ be this type. We claim that $r$ is as required. First, $r$ is  locally isolated by clause (c) above (and $\aleph_1$-isolated by $q$) . By Lemma \ref{lem:isolated_star}, $r \in S_*(B)$, as required.
\end{enumerate}

\end{proof}

\smallskip

Let us define the main property of theories that we are going to investigate.

\begin{de}\label{dfn:fullystable}
We say that $T$ is \emph{fully stable} over $P$ if every complete set $A \subseteq \cC$ is stable over $P$. 
\end{de}

It may be of interest to mention some weaker (more local) notions. We will prove slightly more precise results using these definitions. 

\begin{de}\label{dfn:fullystable2}
\begin{enumerate}
\item 
Given a complete set $A$, we say that $A$ is fully stable over $P$ if every $B \subseteq A$ with $P^A = P^B$ is stable over $P$. 
\item Given $N \models T^P$, we say that $T$ is fully stable over $N$ if every complete set $A$ with $P^A = N$ is stable over $P$. 
\end{enumerate}

\end{de}

\begin{lem}\label{lem:} ($T$ countable)
 	Assume that $A$ is a complete set, and $T$ is fully stable over $P^A$. Then there exists $B \supseteq A$ such that 
	\begin{enumerate}
	\item
		$|B|\le |A|+\aleph_0$
	\item
		Every finite type over $A$ is realized in $B$
		
	\item
		$B$ is locally constructible over $A$
	\end{enumerate}

\end{lem}
\begin{proof}
 	Let $\set{\ph_i(\x,\a_i)\colon i<|A|}$ list all formulae over $A$. Construct an increasing continuous sequence of sets $A_i$ such that:
	\begin{enumerate}	
	\item 
		$A_0 = A$
	\item
		$A_{i+1} = A_i \cup \set{\b_i}$
	\item
		$\models \ph_i(\b_i, \a_i)$
	\item
		$\tp(\b_i/A_i) \in S_*(A_i)$ and is locally isolated
	\end{enumerate}
	This can be done by induction using Lemma \ref{lem:li} for the successor stages, since all $A_i$ are stable. Note that clause (iv) for $b_i, A_i$ implies by induction that all $A_j$ are complete, and $P^{A_j} = P^A$.  
\end{proof}

\begin{re}\label{re:stab_unions}
 Note that since, in the proof above, every $\b_i$ is finite, stability of the set $A_{i+1}$ follows from that of $A_i$. Hence the assumption of stability over a certain $N \models T^P$ can be replaced by the assumption that e.g. stability over $P$ is closed under taking increasing unions. 
\end{re}

We are now ready to prove the main result of this paper. 

\begin{theorem}\label{th:primary}
Assume that $T$ is countable.
\begin{enumerate}
\item 
	Assume that $T$ is fully stable over $N \models T^P$. Then $T$ has the full existence property over $N$ (Definition \ref{de:fullexist}(ii)). 
	
	More specificaly: Let $A$ be a complete set with $P^A = N$. Then there exists $M \models T$ containing $A$ with $P^A = P^M$ and  $|M|\le |A|+\aleph_0$. 
\item 
	Moreover, $M$ in the previous clause is locally constructible (hence $\aleph_1$-constructible) over $A$. 
\end{enumerate}
%
%
%
\end{theorem}
\begin{proof}
	Construct by induction on $i<\om$ sets $A_i$ such that:
	\begin{enumerate}	
	\item 
		$A_0 = A$
	\item
		$|A_i| \le |A|+\aleph_0$
	\item
		Every finite type over $A_i$ is realized in $A_{i+1}$
	\item
		$P^{A_{i+1}} = P^{A_i}$
	\item 
		$A_{i+1}$ is locally constructible over $A_i$. 
	\end{enumerate}
	This is possible by the previous Lemma. Now $M = A_\omega = \bigcup_{i<\om}A_i$ is clearly a model which is locally constructible over $A$, $P^M = P^A$, as required.

\end{proof}


\begin{co}
Assume that $T$ is countable and fully stable over $P$. Then $T$ has the full existence property (Definition \ref{de:fullexist}(i)). In particular, it has the Gaifman property (Definition \ref{de:existence}(iv)). 
\end{co}

Clearly, if the theory $T$ is stable, $P$ any unary predicate in its vocabulary, then $T$ is also stable over $P$. Therefore the following theorem of Lachlan is a consequence of the results above: 

\begin{co}\label{co:lachlan}(Lachlan \cite{Lachlan1972}; see also Remark \ref{re:lachlan} above).
Let $T$ be countable and stable. Then  $T$ has the full existence property. 
\end{co}

A note is in order in regard to the Corollary above. Recall that up until now, we implicitly assumed Hypothesis \ref{asm:1} in all of our proofs. However, this assumption is not explicitly present in Lachlan's paper. If $T$ is stable, then any definable set is (stable and) stably embedded. However, it is not always the case that clause (ii) of Hypothesis \ref{asm:1} holds for $T$ and $P$. So it is not immediately clear that an arbitrary stable theory with an arbitrary distinguished predicate $P$ falls into our framework. However, consider the Morleyzation $T'$ of $T$, and let $\cC'$ be the expansion of the monster $\cC$ of $T$ to the Morleyzation (it is also the monster model of $T'$). Because of the nature of the definition of completeness, Definition \ref{dfn:complete} (which is called ``the TV-property for the pair $(T,P)$'' in \cite{Lachlan1972}), a set $A$ is complete in $\cC$ if and only if it is complete in $\cC'$. Therefore, working in $T'$ (which is also stable, of course), we can deduce the existence property for every complete subset of $\cC'$, and it will imply the same for every complete subset of $\cC$.

\smallskip

At the same time, if one is interested in Gaifman's property for a theory that does not satisfy Hypothesis \ref{asm:1} in full, one needs to be more careful. Indeed, $T^P$ may change if $T$ is Morley-ized; therefore, one can not just work in the Morleyzation, since the notion of a model of $T^P$ depends on whether we work in $T$ or in $T'$. In other words, proving the existence property for all $N \elem \cC^P$ is not necessarily the same as proving it for all $N \models T^P$. This is why Hodges \cite{Hod-cat1}, when working towards a proof of (particular cases of) Gaifman's Conjecture, only considers theories that the ``uniform reduction property'', which is essentially the same as our assumption Hypothesis \ref{asm:1}(ii). Hodges also observes that any theory $T$ relatively categorical over $P$ has this property.

\medskip

We can now connect our results (actually, Lachlan's results) to Hodges' work on relative categoricity of an abelian group over a subgroup. Since any theory of abelian groups (in the language of groups) with an additional predicate for a subgroup is stable, we  obtain the following:

\begin{co}\label{co:groups}
Let $T$ be a theory of abelian groups with a distinguished predicate $P$ picking out a subgroup . Then  $T$ has full existence. In particular, if $T$ also satisfies Hypothesis \ref{asm:1}(ii), then $T$ has the Gaifman property. 
\end{co}

This is, of course, already a consequence of Lachlan's theorem stated above. The following particular case of Gaifman's Conjecture (due to Hodges) follows:

\begin{co}\label{co:hodges}(Hodges \cite{Hod-cat1}).
Let $T$ be a theory of abelian groups with a distinguished predicate $P$ picking out a subgroup, and assume that $T$ is categorical over $P$. Then  $T$ has the Gaifman property. 
\end{co}

We would like to emphasize that most of the results of Hodges, Yakovlev, et al (e.g., \cite{Hod-cat1,Hod-cat2,Hod-cat3}) on abelian groups with a  distinguished subgroup use a very different set of techniques and go in a different direction  (their main objective is understanding \emph{relative categoricity} of $T$ over $P$ for a pair of cardinals $(\kappa, \lambda)$, which is a  more general problem than studying consequences of ``full'' categoricity over $P$ as defined by Gaifman - for example, a vector space over an infinite field is categorical for most, but not all, such pairs; we are hoping to investigate this phenomenon in full generality in future works). This difference in notions of categoricity over $P$ is very much related to \emph{nulldimensionality} defined and discussed in the last section of the present paper.

\bigskip

Finally, compare the main result of this section to the following related theorem from a previous paper of the author:

\begin{fact}(\cite{usvyatsov2024stability})
Let $M \models T$ with $P^M$  a $|T|^+$-saturated model of $T^P$, and assume that $M$ is fully stable over $P$. Then $M$ has the full existence property.
\end{fact}

\section{Full stability and $ACFA_0$}

In this section we discuss the notion of full stability over $P$, and remark that the theory $T = ACFA_0$ falls into this category. We assume familiarity with the basics of ``stability-like'' theory of $ACFA_0$ as developed in \cite{Chatzidakis1999ModelTO} (see also \cite{Chatzidakis2020RemarksAT}) for a concise summary). 

\medskip

 Assume for simplicity that for some ``big enough'' $\lam$ we have $\lam^{<\lam} = \lam$ (for example, $\lam = \mu^+ = 2^\mu$). So every theory $T$ of cardinality $<\lam$ has a saturated model in $\lam$. 

\begin{lem}\label{lem:satstab}
Let $M_P\models T^P$ be saturated of cardinality $\lam$ as above, and assume that for every complete $B \models T$ with $P^B = M_P$ we have $|S_*(B)| < 2^{|B|}$. Then $T$ is fully stable. 

\end{lem}
\begin{proof}
 Let $A$ be a complete set. Expand the language by a new predicate $Q$ interpreted as $A$, call the expanded theory $T'$. Let $N'$ be a saturated model of $T'$ of cardinality $\lam$ and let $N\models T$ be its restriction to $L$. Let $B = Q^{N'}$. 
 
Note:
 
 \begin{itemize}
\item $B$ is a saturated model of $\Th(A)$.
\item Hence $P^B$ is a saturated model of $T^P$ (and so is isomorphic to $M_P$).
\end{itemize}

The second item is true because $A$ is complete. 

By the assumption of the Lemma,  $|S_*(B)| < 2^{|B|}$. By Fact \ref{fct:satstable}, $B$ is stable (hence so is $A$).

 
\end{proof}

%
%

From now on, let $T$ be a completion of $ACFA_0$, $\cC$ the monster model of $T$, $P = Fix(\sigma)$. Recall that $P$ is stably embedded (\cite[1.11]{Chatzidakis1999ModelTO}), $T$ eliminates imaginaries (\cite[1.10]{Chatzidakis1999ModelTO}). In the the most natural language (of fields with a function symbol for the automorphism), $T$ is model complete, but does not eliminate quantifiers (\cite{Macintyre1997GenericAO}). In order to ``fit'' $T$ into our context, we need to replace all function symbols by predicate symbols (and add appropriate axioms to the theory), and add predicate symbols to ensure quantifier elimination (the latter is not absolutely necessary, since Hypothesis \ref{asm:1}(ii) holds regardless of quantifier elimination - see Proposition 1.11 in \cite{Chatzidakis1999ModelTO}). For example, we can expand the language with the predicates used in Macintyre \cite{Macintyre1997GenericAO}, Theorem 12. 


Since $T = T^{eq}$, we can use the following characterization obtained Theorem \ref{thm:eq}:

\begin{obs}\label{obs:ACFeq}
 Let $A \subseteq \cC$. Then $A$ is complete if and only if  $dcl(A) \cap P^\cC = P^A$ and $P^A \prec P^\cC$ if and only if  $acl(A) \cap P^\cC = P^A$ and $P^A \prec P^\cC$. 
 \end{obs}
\noindent In particular:

\begin{re}
 Let $A$ be a difference field of characteristic 0 with $P^A$ pseudo-finite. Then $A$ is (can be seen as) a complete subset of the monster model of $T$ as above. 

\end{re}

\begin{re}\label{re:acl}
\begin{enumerate}
\item A set $A$ is complete if and only if $P^A = P^{\acl(A)}$ and $\acl(A)$ is complete. In particular, if $A$ is complete, then for every $b \in \acl(A)$, $\tp(b/A) \in S_*(A)$. Similarly, $\tp(a/A) \in S_*(A)$ if and only if $\tp(a/\acl(A)) \in S_*(\acl(A))$.  
\item Let $A$ be complete, $a,b$ tuples. Then $\tp(ab/A) \in S_*(A)$ if and only if 
\begin{itemize}
\item $\tp(a/A) \in S_*(A)$ (hence $Aa$ is complete), and
\item $\tp(b/Aa) \in S_*(Aa)$
\end{itemize}

\end{enumerate}

\end{re}

\begin{ex}
	Note that if $A \subseteq \cC$ and $a_0,a_1  \in A$ such that $\sigma(a_i) = \lam\cdot a_i$ for some $0 \neq \lam\in A$, then $b = \frac{a_0}{a_1} \in P$. Hence if $b \notin P^A$, then $A$ is incomplete (even if $P^A \prec P^\cC$).
	
	In this case, it may be the case that $p = \tp(a_1/A) = \tp(a_0/A) \in S_*(A)$, but $\tp(a_1/Aa_0) \notin S_*(Aa_0)$, and $\tp(a_1a_0/A) \notin S_*(A)$.
\end{ex}


\begin{obs}
 Let $A$ be a complete set, $p \in S_*(A)$. Then $p$ is weakly orthogonal to the fixed field in the sense of $ACFA_0$; i.e., for every model $M\models T$ containing $A$ and every $a\models p$ in $M$, we have $a \ind_A P^M$.  
\end{obs}
\begin{proof}
 If $p\in S_*(A)$, then, by the definition, for any such $M$ and $a\models p$, the type $\tp(a/P^M)$ is definable over $P^A$, hence it does not fork over $P^A$. 
%
\end{proof}

Note that the converse is also true. Indeed, if $\tp(a/A)$ as above is weakly orthogonal to $P$, then $a \ind_A P^M$, hence $\acl(Aa) \ind_A P^M$, which, in this case, means that $\acl(Aa)$ and $P^M$ are linearly free over $\acl(P^A) = P^A$ (see \cite[2.4]{Chatzidakis2020RemarksAT}). In particular $\acl(Aa)\cap P^M \subseteq A$, which implies $\tp(a/A) \in S_*(A)$ by Observation \ref{obs:ACFeq}.

%
%

\begin{theorem}\label{thm:ACFA}
 $T = ACFA_0$ is fully stable over $P = Fix(\sigma)$.
\end{theorem}

The main tool in our proof is the Chatzidakis-Hrushovski ``semi-minimal'' analysis of types of finite rank \cite{Chatzidakis1999ModelTO}. We refer the reader to \cite{Chatzidakis2020RemarksAT} for a review of the relevant notions, specifically, $SU$-rank and qf-internality (quantifier free internality) to $P$ (Definition 2.12 in \cite{Chatzidakis2020RemarksAT}). 

\begin{fact}\label{fact:decomposition} (\cite{Chatzidakis1999ModelTO}, Theorem 5.5) Suppose that $SU(a/E) < \omega$, where $E$ is a difference field. Then there is a sequence $a_1, \ldots,a_n$ of elements in $\acl(Ea)$ such that $a \in \acl(Ea_1 \ldots  a_n)$ and, for every $i < n$, $\tp(a_{i+1}/\acl(Ea_1\ldots a_i))$ is either stable modular of $SU$-rank 1, or qf-internal to the fixed field $P$.
\end{fact}

The Semi-minimal Analysis Theorem suggests that a reasonable first step towards proving stability would be counting the ``basic building blocks'' appearing in the decomposition. We would like to thank Zoe Chatzidakis for pointing out to us Lemma 3.4 in \cite{Chatzidakis2020RemarksAT}, which will be the central technical tool in the analysis of one kind of these building blocks: types qf-internal to the fixed field. 

\begin{lem}
 \label{lem:rankonestable}
 Let $A$ be a complete set with $P^A$ an $\aleph_1$-saturated model of $T^P$. 
 
 Denote  

\[ 
ST(A) = \set{p \in  S_*(A) \colon p \text{ is stable }} 
 \]
 
and
 
 \[ 
IN(A) = \set{p \in  S_*(A) \colon p \text{ is qf-internal to  } P} 
 \]

Then $|ST(A)| + |IN(A)| \le |A|^{\aleph_0}$.
\end{lem}
\begin{proof}
	By Remark \ref{re:acl}, $\acl(A)$ is complete, and $P^{\acl(A)} = P^A$. So without loss of generality $A = \acl(A)$.
	
 	The statement is clear for $ST(A)$, since stable types over $A$ are definable over $A$ (but see also Remark \ref{rmk:superstable} below). 
	
	Now let $\tp(a/A)$ be qf-internal to $P$. By Lemma 2.14 \cite{Chatzidakis2020RemarksAT}, every (realization of a) type over $A$, which is qf-internal (and indeed almost internal) to the fixed field is inter-algebraic with a realization $a'$ of a type qf-internal to $P$ satisfying in addition $\sigma(a') \in \dcl(Aa')$. Hence without loss of generality $\sigma(a) \in \dcl(Aa)$. In addition, by Lemma 3.2 in \cite{Chatzidakis2020RemarksAT}, since $P^A$ is $\aleph_1$-saturated, there is an $\aleph_1$-saturated model $M\models T$ with $P^A = P^M$. 
	Now, all the assumptions of Lemma 3.4 in \cite{Chatzidakis2020RemarksAT} hold for $M$, $A$, and $a$. Indeed, $M\models ACFA_0$ is $\aleph_1$-saturated (hence $\aleph_\eps$-saturated), $A = \acl(A)$ with $P^M \subseteq A$, $p = \tp(a/A) \in S_*(A)$, hence $p \perp_a P^M$; $p$ is qf-internal to $P$, and $\sigma(a) \in \dcl(Aa)$. Therefore, the conclusion of \cite[3.4]{Chatzidakis2020RemarksAT} holds as well.  Specifically, there is a ``very small'' $B \subseteq A$ such that $\tp(a/B) \vdash \tp(a/A)$, where ``very small'' means that there is a finite $B_0$ such that $B = \acl(B_0)$.; i.e., $p$ is $\aleph_\eps$-isolated (see Definition 2.16 in \cite{Chatzidakis2020RemarksAT}).
	
	Therefore, $|ST(A)|$ is bounded above by the number of $\aleph_\eps$-isolated types over $A$, which is determined by the number of types over finite subsets of $A$. Since the number of such subsets is $|A|^{<\aleph_0} = |A|$, we conclude $|ST(A) \le |A|+2^{\aleph_0}$.

\end{proof}

\begin{remark}\label{rmk:superstable}
	Recall  that in $ACFA_0$ stable types are, in fact, ``superstable''; that is, every stable type is definable over a ``very small'' set (see \cite[2.6]{Chatzidakis2020RemarksAT}). This implies that the number of stable types over $A$ is also bounded above by $|A|+2^{\aleph_0}$; i.e., in the Lemma above, $|ST(A)| + |IN(A)| \le |A|+2^{\aleph_0}$.
\end{remark}

We can now conclude the proof of the Theorem. 

\medskip

\begin{proof}
 (of Theorem \ref{thm:ACFA}). 
 
 We work with a saturated model $M \models T$. By Lemma \ref{lem:satstab}, it is enough to prove that, for every complete $A$ with $P^A = P^M $, the number of types in $S_*(A)$ is ``not big''. 
 
 Let $A$ be complete with $P^A = P^M$. So $P^A$ is a saturated pseudo-finite field. First, Remark \ref{re:acl} allows us to make a few simplifications:
 
 \begin{itemize}
\item Without loss of generality, $A = \acl(A)$. 
\item It is enough to consider $S_*^1(A)$,  the set types $\tp(a/A) \in S_*(A)$ where $a$ is a singleton. This is because the number of $n$-types of the form $p = \tp(a_0a_1 \ldots a_{n-1}/A) \in S_*(A)$ is bounded by $\Sigma_{i<n} S_*^1(\acl(Aa_{<i}))$, where $a_{<i} = \lseq{a}{j}{i}$, and, since $p$ is complete, all the sets $B_i = \acl(Aa_{<i})$ have the same $P$-part, which equals $P^A$. Hence proving that the number of relevant 1-types over all sets $A'$ with $P^{A'} = P^M$ is ``small'', would imply the same for all $n$-types.

\end{itemize}

Since there is a unique generic 1-type (1-type of $SU$-rank $\omega$) over $A$ (see \cite[2.7]{Chatzidakis2020RemarksAT}), it is enough to count types of finite $SU$-rank. 

By the Decomposition Theorem (Fact \ref{fact:decomposition}), any (realization $a$ of a) type of finite rank is interalgebraic with a tuple $a_0a_1 \ldots a_{n-1}$ such that the type of each $a_i$ over $B_i = \acl(Aa_{<i})$ is either stable, or qf-internal to $P$. Since $\tp(a/A) \in S_*(A)$, so is the type $\tp(a_{<n}/A)$. By Remark \ref{re:acl} again, $\tp(a_i/B_i) \in S_*(B_i)$, and $P^{B_i} = P^A$ for all $i$. Therefore, just as before, in order to bound the number of types in $S_*(A')$ of finite rank for all $A'$ with $P^{A'} = P^M$, it is enough to do this for such types that are either stable over $A'$, or are qf-internal to $P$. And this is exactly the content of Lemma \ref{lem:rankonestable}.  

\end{proof}

\begin{remark}
In Remark \ref{rmk:superstable} we have observed that the number of ``semi-minimal'' types of finite rank over any complete set $A$ with a slightly saturated $P$-part is, in fact, at most $|A| + 2^{\aleph_0}$. The proof of the Theorem above implies that the same is true for $S_*(A)$ for such a set $A$, i.e.,  $T$ is, in a sense, ``fully superstable'' over $P$. This observation should become important once the theory of superstability over $P$ is developed. In particular, superstability should become handy in the study of complete subsets of $T$ with an $\aleph_\eps$-saturated $P$-part. One goal of such study would be obtaining a common generalization of results in \cite{Chatzidakis2020RemarksAT} and \cite{usvyatsov2024stability} for theories fully superstable over $P$. We intend to return to this in a future work. 
\end{remark}

\section{On nulldimensionality over $P$}

Recall that Gaifman's Conjecture says that a countable theory categorical over $P$ has the Gaifman property. Here we have shown that the Gaifman property follows from full stability over $P$, and established the existence of (interesting) fully stable theories. However, the strength of the assumption of full stability (in the context of Shelah's Classification  Theory) is not clear. For example, we do not know whether categoricity over $P$ implies full stability. 

In \cite{Sh234}, Shelah has showed that if a countable $T$ ``has absolutely no two-cardinal models'' and every ``u.l.a (uniformly locally atomic) $\mathcal{P}^-(n)$-system'' in $T$ is stable, then $T$ has  the Gaifman property (and moreover, it has the existence property over any u.l.a. $\mathcal{P}^-(n)$-system; a model of $T^P$ being a 0-dimensional particular case).

 It is probably self-explanatory what Shelah means by a ``two-cardinal model''  (but we give a definition below). We shall not define the notions of uniform local atomicity or a $\mathcal{P}^-(n)$-system here; for the purpose of this discussion, it is enough to just say that it is a complete set of a very particular form, resulting from ``stable amalgamation'' of models of $P$ and models of $T$. We refer the reader to the source \cite{Sh234} for details. Alternatively, the thesis of Anass Alzurba \cite{Alz-thesis}  contains an excellent and detailed exposition of Shelah's proof that, assuming that $T$ is countable, has absolutely no two-cardinal models, and every u.l.a system is stable, $T$ has the Gaifman property (and indeed the extension property for u.l.a. systems). 

Shelah's theorem stated above is combination of several difficult and important results, although it does not seem to fully establish Gaifman's Conjecture. The first assumption of the theorem (absolutely no two-cardinal model) follows  from categoricity over $P$ (at least for a countanble $T$). The second one follows only modulo some set-theoretic assumptions. Specifically, Shelah proves (in the same paper) a series of  \emph{consistent} non-structure results: if some u.l.a $\mathcal{P}^-(n)$-system in $T$ is unstable, then for every uncountable $\lambda = \lambda^{<\lam}$, there is a forcing notion that preserves cardinals $\le \lam$, and in $V[G]$ (the forcing extension) $T$ has non-structure in $\lam$, i.e., $2^\lam$ models with the same $P$-part, all pairwise non-isomorphic over $P$. 

So we do not (for now) seem to have a  full proof for Gaifman's Conjecture, as originally stated, although Shelah's Theorem comes very close (and potentially some tweaks in the non-structure arguments can lead to a complete solution). 

At the same time, in the context of model theory, it makes sense to consider properties that are absolute. Some questions related to constructions of models over a predicate, can easily ``slide'' into the realm of set theory. Chang's Conjecture is a very well-known and extensively studied example; but it is quite possible that related set theoretic phenomena affect properties such as the existence property and the number of non-isomorphic models over $P$. These are potentially very interesting problems and phenomena, but of less relevance to us here. Therefore, perhaps the notions of ``absolute categoricity'' and ``absolutely no two-cardinal models'' defined below,  anyway lead to the more ``correct'' questions in our context, and, ultimately, to results of more model theoretic flavor. Similarly, perhaps it makes  sense to aim for ``absolute structure results'', allowing seeking non-structure in forcing extensions (as is done in \cite{Sh234}).

\medskip

So in spite of the fact that the non-structure results in \cite{Sh234} are established using forcing, they still give a very powerful non-structure theory. At present we have no analogous results (and indeed no non-structure theory) for the failure of full stability. Hence the ``classification theory strength'' of Shelah's stability assumption may turn out to be significantly weaker than full stability defined here. 

On the other hand, the first assumption of Shelah's theorem (absolutely no two cardinal model) is very strong. In this section we will briefly discuss this assumption and observe that it is equivalent to a more intrinsic  (in the context of model theory over $P$) notion, which we will refer to as \emph{nulldimensionality} over $P$. The idea is that there is not a single dimension that can be varied without affecting $P$. More formally, no type over a model is orthogonal to $P$; i.e., for every model $M \models T$, every type $p \in S_*(M)$ is algebraic, i.e., realized in $M$. 

This is obviously a very significant restriction. It makes perfect sense when one investigates categoricity over $P$ as defined by Gaifman, but it does not hold in most examples. In fact, it breaks down as soon as one allows oneself to consider slightly weaker notions of categoricity over $P$ like the ones investigated by Hodges, e.g.,\cite{hodges-gaifman} (more familiar notions of categoricity in power, satisfied, for example, by the class of vector spaces over a  field). No example mentioned earlier in this paper is nulldimensional. In particular, in $ACFA_0$, we have the generic type (the type of rank $\omega$) over every model $M$, which is, of course, non-algebraic (and is orthogonal to $P = Fix(\sigma)$).

\medskip

Let us now make the discussion about a little more formal. 

\begin{de}
 \begin{enumerate}
 \item We say that $M_1,M_2 \models T$ are isomorphic over $P$ if there exists an isomorphism from $M_1$ onto $M_2$ which is the identity on $P$ (so in particular, $P_{M_1} = P_{M_2}$). We write $M_1 \cong_{P} M_2$. 
\item We say that $T$ is \emph{categorical over $P$} is for every $M, M' \models T$, $P^{M} = P^{M'} \implies M \cong_P M'$.
\item We say that $T$ is \emph{absolutely categorical over $P$} if $T$ is categorical over $P$ in every forcing extension of the ground universe $V$.
\end{enumerate}

\end{de}

So Gaifman's Conjecture states that if a countable $T$ is categorical over $P$, then it has the Gaifman property. Shelah's results mentioned above establish Gaifman's Conjecture under the assumption of absolute categoricity:

\begin{theorem}(Shelah \cite{Sh234}).
Let $T$ be countable and absolutely categorical. Then it has the Gaifman property. 
 
\end{theorem}

This theorem is never stated in writing. But it follows from the combination of the structure and the non-structure results mentioned above. Assuming ``absolutely no two-cardinal models'', Shelah proves, on the one hand, Gaifman's property assuming stability over $P$ of all u.l.a $\mathcal P^-(n)$-systems, and on the other, he establishes consistent non-structure results assuming the existence of an unstable u.l.a $\mathcal P^-(n)$-system.

\begin{de}\label{de:2card}
\begin{enumerate}
\item 
 We say that $M \models T$ is a \emph{two-cardinal model of $T$ over $P$}, or just a \emph{two-cardinal model} if $|T|\le |P^M| < |M|$.
 \item 
 Given two cardinals $\ka \le \lam$, we say that $M \models T$ is a model of type $(\ka,\lam)$ (over $P$) if $|P^M| = \ka$ and $|M| = \lam$. 
 \item
 We say that $T$ has absolutely no two-cardinal models if no $T_0 \subseteq T$ finite has a two cardinal model.
\item 
We say that $(M,N)$ is a \emph{Vaughtian pair} over $P$ if:
\begin{itemize}
\item $M \elem N \elem \cC$
\item $M \neq N$
\item $P^M = P^N$
\end{itemize}
\end{enumerate}

\end{de}

%

\begin{re}
 If $T$ has absolutely no two-cardinal models, it has no two-cardinal models. As Shelah points out in \cite[1.2]{Sh234}, $T$ has absolutely no two-cardinal models if and only if $T$ has no Vaughtian pair over $P$. 
\end{re}

\begin{re}
If $T$ is categorical over $P$, it has no two-cardinal models of type $(\ka,\lam)$ with $|T| \le \ka$. If $T$ is countable (and $P$ infinite), and $T$ is categorical over $P$, then $T$ has no two-cardinal models. 
\end{re}

\smallskip

Next we observe that we have something like ``generalized Vaught's test'' for completeness of theories: every two saturated models with the same $P$-part are isomorphic over $P$. This is implicit in \cite{Sh234}, but we elaborate in order to make this section more self-contained. 

\begin{lem}\label{lem:satprimary}(Fact 1.7 in \cite{Sh234})\label{lem:aleph1-sat}
 Let $M\models T$, $a \in M$. Then $p=\tp(a/P^M)$ is $|T|^+$-isolated, and $\tp(a/P^M) \vdash \tp(a/P^\cC)$.
\end{lem}
\begin{proof}
 Let $B \subseteq P^M$, $|B| = |T|^+$, be such that $\tp(a/P^\cC)$ is definable over $B$. More precisely, for every formula  (without parameters)  $\ph(x,y)$ there is $c\in B$ such that the formula $\forall y \in P (\ph(x,y) \iff \Psi_\ph(y,c)$ is in $p$ (compare to  Corollary \ref{co:type}; note that we do not need additional parameters from $P$ in the formula). If $a\equiv_B a'$ (or $a'$ just satisfies all the formulas above), then for every formula $\ph(x,y)$ and $b \in P$, we have $\ph(a,b) \iff \Psi_(b,c) \iff \ph(a',b)$, as required.  
 
 \end{proof}

\begin{co}\label{co:modeliso}
\begin{enumerate}
\item Let $M \models T$, and let $\a \in M$ be a tuple of length $\ka$. Then $\tp(\a/P^M)$ is $(\ka + |T|^+)$-isolated. 
\item  Let $M \models T$, $A \subseteq M$, $|A| = \ka$, $a \in M$ finite, then $\tp(a/P^M\cup A)$ is $(\ka + |T|^+)$-isolated.
\end{enumerate}
\end{co}

%

\begin{co}(Theorem 1.5 in \cite{Sh234})
 Let $M \models T$ saturated of cardinality $\lam > |T|$. Then $M$ is $\lam$-primary over $P^M$ (see Definition \ref{dfn:lprimary}).
\end{co}
%
%

%

\begin{theorem}\label{thm:satiso}
\begin{enumerate}
\item
	Let $M \models T$ be saturated of cardinality $\lam>|T|$, and $M'$ of cardinality $\lam$ with $P^{M'} = P^M$. Then there exists an elementary embedding $f \colon M' \hookrightarrow M$ over $P$ (i.e., $f\rest P = id$). Moreover, if $A \subseteq M'$, $|A|<\lam$,  $f \colon P^M \cup A \hookrightarrow M$ is a partial elementary embedding, and $a \in M'$, then $f$ can be extended to $P^M \cup A \cup \set{a}$. 
\item 
 	Let $M, M$ be saturated of the same  cardinality $\lam > |T|$ with $P^{M} = P^{M'}$. Then $M \cong_P M'$. 
\end{enumerate}
\end{theorem}
\begin{proof}
\begin{enumerate}
\item Let $N = P^{M} = P^{M'}$. It is enough to prove the last statement, so let $A, a$ be as there. By Corollary \ref{co:modeliso}(ii), the type $\tp(a/N\cup A)$ is $\lam$-isolated. So there is $C \subseteq N$ such that $\tp(a/C \cup A) \vdash \tp(a/N \cup A)$. Since $M$ is saturated, we can realize $f(\tp(a/C \cup A))$ in $M$, as required. 


\item   Follows from (i) by a back and forth argument. 

\end{enumerate}

\end{proof}

\medskip

We now define the notion of nulldimensionality and relate it to Vaughtian pairs, two-cardinal models, and categoricity. 

\begin{de}
 We say that $T$ is \emph{nulldimensional} over $P$ if for every $M \models T$, every $p \in S_*(M)$ is algebraic (i.e., realized in $M$).
\end{de}

\begin{theorem}\label{thm:nulldim}
 Let $T$ be any theory, and let $\lam$ be such that $|T|<\lam = \lam^{<\lam}$. 
Then the following are equivalent:
\begin{enumerate}
\item $T$ is not nulldimensional over $P$.
\item For some saturated $M \models T$, we have a non-algebraic $p \in S_*(M)$.
\item $T$ has a Vaughtian pair $(M,N)$ over $P$.
\item $T$ has a Vaughtian pair $(M,N)$ over $P$ with $P^M = P^N$ saturated of cardinality $\lam$.
\item $T$ has a two-cardinal model of cardinality $>|T|$.  
\item $T$ has a two-cardinal model with $P^M$ saturated. 
\item $T$ has a two-cardinal model of type ($\lam, \lam^+$) with $P^M$ saturated.  
\item $T$ has a two-cardinal $\lam$-saturated model of type ($\lam, \lam^+$).  
\end{enumerate}
\end{theorem}
\begin{proof}
Let $T$ be non-nulldimensional, and let $M, a$ witness this (that is, $\tp(a/M) \in S_*(A)$ non-algebraic. Let $(M',a') \equiv (M,a)$ saturated of cardinality $\lam$ (so we add a new predicate for $M$ and take a saturated model of the theory of $A = M\cup\set{a}$ in the expanded language). In particular, the reduct of this model to the original language satisfies that $P^{M'}$ is saturated, and $A' = M'\cup\set{a'}$ is complete, since $A' \equiv A$. By Fact \ref{fact:satexist} there is a model $M''$ with $P^{M''} = P^{M'}$ containing $a'$; so (i) $\implies$(iv). 

A  similar (and easier) argument shows that (iii) $\implies$ (iv). Obviously, (iv) $\implies$ (iii), (iv) $\implies$ (ii), and (ii) $\implies$ (i). This proves that (i) -- (iv) are equivalent. So far we have not used that there exists $\lam = \lam^{<\lam}$ (although the assumption that $T^P$ has saturated models will normally imply this condition anyway). 

 Clearly, (viii) implies (vi), (vii) implies both (vi) and (v), which, in turn, imply (i). We now prove that (i) implies (viii). As in the proof of (i) $\implies$ (iv), we can find a Vaughtian pair $M_0,M_1$ of models of cardinality $\lam$ with $M_0$ saturated. Moreover, since $\lam = \lam^{<\lam}$, we may assume that $M_1$ is also saturated. By Theorem \ref{thm:satiso}, $M_0, M_1$ are isomorphic over $P$. Hence there is $M_2$ saturated such that $(M_1, M_2)$ form a Vaughtian pair. This argument allows us to construct by induction a strictly increasing sequence \lseq{M}{i}{\lam^+} of saturated models with $P^{M_i} = P^{M_0}$ (for $\delta$ limit of cofinality $<\lam$, note that the union $M'_\delta = \cup_{i<\delta} M_i$ is still a model, hence a complete set with a saturated $P$-part; by Fact \ref{fact:satexist}, it can be extended to a saturated model $M_\delta$ without changing $P$). The union of this chain of models is the desired two-cardinal model. Note that a union of length $\lam^+$ of $\lam$-saturated models is itself $\lam$-saturated, as required.



\end{proof}

\begin{re}
 The argument for (i) $\implies$ (viii) is  a particular case of Chang's two-cardinal theorem (our assumptions make our lives easy; note that e.g., Hypothesis \ref{asm:1} is used implicitly in our proof).
\end{re}

\begin{re}\label{re:absolute}
 The notion of nulldimensionality is absolute. Moreover, since it is equivalent to not having a Vaughtian pair, by Discussion 1.2 in \cite{Sh234}, it is also equivalent to Shelah's notion of ``absolutely no two-cardinal model''. 
\end{re}

\begin{de}
\begin{enumerate}
\item 
A subclass $K$ of models of $T$ is called categorical over $P$ if  for every $M_1,M_2 \in K$,  $P^{M_1} = P^{M_2} \implies M_1 \cong_P M_2$. 
\item A subclass $K$ of models of $T$ is called absolutely categorical over $P$ if for every forcing extension $V[G]$, the class $K$ (as interpreted in $V[G]$) is categorical over $P$. 
\end{enumerate}
 
\end{de}

\begin{re}
 What we mean in clause (ii) above is really that $K$ is a template (a definition) of a class of models, not necessarily the class itself. For example, $K$ can be the class of saturated (or $|T|^+$-saturated) models of $T$; this class can have different interpretations in different models of set theory (for example, it is possible that $T$ has no saturated models in $V$, or that a model that appears to be saturated in $V$ stops being saturated in a forcing extension). 
\end{re}

\begin{co}\label{co:satcat}
Assume that for some cardinal $\lam>|T|$ we have $\lam = \lam^{<\lam}$. 
Let $K$ be the class of all models of $T$ with $P^M$ saturated of cardinality $\lam$ for some $\lam = \lam^{<\lam}$. 
Then TFAE:
\begin{enumerate}
\item $T$ is nulldimensional.
\item For $M\models T$  if $P^M$ is saturated of cardinality $\lam$, then so is $M$.
\item $K$ is  categorical over $P$.
\end{enumerate}

\end{co}
\begin{proof}
(i) $\implies$ (ii): Assume $T$ is nulldimensional, and let  $M \models T$ be such that $P^M$ is saturated of cardinality $\lam = \lam^{<\lam} > |T|$. By Fact \ref{fact:satexist}, there exists $M' \models T$ saturated of cardinality $\lam$ with $P^M = P^{M'}$. Note that since $T$ is nulldimensional, $|M| = \lam$. By Theorem \ref{thm:satiso}(i), $M$ can be elementarily embedded into $M'$ over $P$. By nulldemensionality, $M = M'$. 

(ii) $\implies$ (iii) by Theorem \ref{thm:satiso}(ii). 

(iii) $\implies$ (i) By Theorem \ref{thm:nulldim}(vii).

\end{proof}

Assuming  stability, a stronger characterization can be established. The following theorem shows that, assuming full stability, nulldimensionality can, in some sense, be seen as a certain analogue (over $P$) of unidimensionality in classical model theory (at least in relation to questions of categoricity). 


\begin{lem}\label{lem:nullthenpluscat}
 If the class of $T^+$-saturated models is absolutely categorical over $P$, then it is nulldimensional. 
\end{lem}
\begin{proof}
 Working in some forcing extension that does not change sets of cardinality $\le |T|^+$, but where there is $\mu > |T|^+$ with $2^{\mu} = \mu^+$, we get that $\lam = \mu^+$ satisfies the assumption in Theorem \ref{thm:nulldim}. If $T$ is not nulldimensional, clause (viii) of Theorem \ref{thm:nulldim} (together with Fact \ref{fact:satexist}) implies that that there are two $\lam$-saturated models with the same $P$-part of different cardinalities; so, since $\lam > |T|$,  the class  of $|T|^+$-saturated models of $T$ is not absolutely categorical.

\end{proof}

\begin{theorem}
 Let $T$ be fully stable. Then $T$ is nulldimensional if and only if the class of $|T|^+$-saturated models of $T$ is absolutely categorical over $P$.
\end{theorem}
\begin{proof}
 Assume $T$ is nulldimensional, and let $M \models T$ be $|T|^+$-saturated (potentially in some forcing extension $V[G]$). By Theorem 4.6 in \cite{usvyatsov2024stability}, there is (in $V[G]$) $M_0 \elem M$, $P^{M_0} = P^M$, $M_0$ is $|T|^+$-primary over $P^M$. Since $T$ is nulldimensional, $M_0 = M$. Hence, if $M' \models T$ (again, in $V[G]$) is $|T|^+$-saturated with $P^{M'} = P^M$, there is an elementary embedding $f$ of $M$ into $M'$ over $P$. Again, by nulldimensionaly, the image of $f$ is $M'$, so it is an isomorphism over $P^M$. 
 
The other direction is Lemma  \ref{lem:nullthenpluscat} (and does not use stability). 
 
\end{proof}


\begin{re}
 In an upcoming preprint \cite{Us25}, the author improves the result above, replacing categoricity for the class of $|T|^+$-saturated models with categoricity over a $|T|^+$-saturated predicate.  
\end{re}

Using Shelah's results from \cite{Sh234}, one can weaken the stability assumptions for countable theories. Again, the reader is referred to \cite{Sh234} or \cite{Alz-thesis} for definitions of the relevant notions and further details.

\begin{theorem}
 Let $T$ be countable, and assume that every u.l.a. $\mathcal P^-(n)$-system in $T$ is stable over $P$. Then $T$ is nulldimensional if and only if the class 
os $\aleph_1$-saturated models of $T$ is absolutely categorical.
\end{theorem}
\begin{proof}
By Theorem 6.2 in \cite{Sh234} (see also Theorems 3.42 and 8.7 in \cite{Alz-thesis}), if $T$ is countable and nulldimensional, then over every $N\models T^P$ there is a locally constructible model $M_0 \models T$. In particular, such $M_0$ is $\aleph_1$-constructible over $N$ (see definitions in section 6 and Observation \ref{obs:localimpliestplus}), therefore can be elementarily embedded over $N$ into any $|T|^+$-saturated model of $T$. So one can repeat the proof of the previous Theorem using this $M_0$ instead of the primary model used there. 
\end{proof}

\medskip


\begin{lem}
 \begin{enumerate}
\item If $T$ satisfies Hypothesis \ref{asm:1}, then $N$ is stable over $P$ for every $N \prec \cC^P$. 
\item If $T$ is absolutely categorical over $P$, then $T$ is nulldimensional. 
\item If $T$ is countable and categorical over $P$, then it is nulldimensional. 
\item If $T$ is nulldimensional, then  $M$ is stable over $P$ for every $M \prec \cC$.
\item If some (equivalently, every) $M \prec \cC$ is unstable over $P$, then there is a forcing extension $V[G]$ and a saturated $N\models T^P$ of cardinality $\lam = \lam^{<\lam} > |T|$, such that
there are $2^\lam$ models $M_i \models T$ pairwise non-isomorphic over $P$ with $P^{M_i} = N$. 

In particular, $T$ is not absolutely categorical over $P$ (but in this case, this follows already from clauses \emph{(ii)} and \emph{(iv)}).
\item Same conclusion if $T$ is countable, nulldimensional, and there is an unstable u.l.a. $\mathcal P^-(n)$-system (in particular, $T$ is not absolutely categorical over $P$).
\end{enumerate}

\end{lem}
\begin{proof}
 \begin{enumerate}
\item 
Follows almost immediately from Fact \ref{fact:satexist} and Theorem \ref{thm:satiso}(ii). But there is also a straightforward counting argument: if $N \prec \cC^P$ saturated, let $p \in S_*(N)$, and let $a \models p$. Since $p \in S_*(N)$, the set $B = N \cup \set{a}$ is complete. So $\tp(a/P^\cC)$ is definable over $P^B = N$. In particular, $p$ is definable over $N$. So all types in $S_*(N)$ are definable over $N$, hence there are at most $|N|^|T|$ such types. 
%
\item Easy (use the ``no two-cardinal model'' criterion).
\item Same as previous clause. 
\item If some model $M$ of $T$ is unstable, then for some model $M'$ of $T$, we have $|S_*(M')|>|M'|$. In particular, there is a non-algebraic type in $S_*(M')$.

%


\item See section 2 of \cite{Sh234}.

\item Scattered through \cite{Sh234} (section 2 for $n=2$, section 4 for $n=3$, and section 9 for $n>3$). 

\end{enumerate}

\end{proof}

Let us  summarize what we know about the connections between nulldimensionality, stability, and categoricity. 
\begin{co}\label{co:nullcatequiv}
 \begin{enumerate}
\item If $T$ is categorical over $P$, then $T$ is nulldimensional, every model of $T^P$, and every model of $T$ are stable over $P$.  
\item If $T$ is absolutely categorical over $P$, then $T$ is nulldimensional, and every u.l.a. $\mathcal P^-(n)$-system is stable over $P$ (but see the Discussion below).  
\item If $T$ is nulldimensional and fully stable, then the class of $|T|^+$-saturated models $M$ of $T$ 
is absolutely categorical over $P$. 
\item If $T$ is countable, nulldimensional, and every u.l.a. $\mathcal P^-(n)$-system is stable over $P$, then the class of $\aleph_1$-saturated models of $T$ is absolutely categorical over $P$.
\item If the class of $|T|^+$-saturated models of $T$ 
is absolutely categorical over $P$, then $T$ is nulldimensional, every model of $T^P$, and every model of $T$ are stable over $P$.  
\item If $T$ is countable and the class of $\aleph_1$-saturated models is absolutely categorical, then $T$ is nulldimensional, and every u.l.a. $\mathcal P^-(n)$-system is stable over $P$ (at least for $n\le 3$).  
\end{enumerate}

\end{co}
\begin{proof}
 The only statement that was not  stated explicitly earlier is clause (vi), which, again, follows from Shelah's non-structure theorems (see the discussion below).
\end{proof}

\begin{dsc}
 Some clarification is in order regarding the last clause of Corollary \ref{co:nullcatequiv}. In Theorem 4.3 of \cite{Sh234}, it is explicitly stated and proven, that, if $n=3$, and there is an unstable $\mathcal P^-(3)$-system, then one can force the existence of ``many'' $\lam$-saturated models (for some $\lam>|T|$) non-isomorphic over $P$. However, in Theorem 9.2 in \cite{Sh234}, where the non-structure result is stated for an arbitrary $n$, saturation is not mentioned explicitly, and the proof is very much incomplete. So really in clause (vi) above only the cases for $n \le 3$ are fully proven in \cite{Sh234}. However, we believe that the proof of \cite[9.2]{Sh234} does indeed go through as claimed, and that it, too, gives non-structure for the class of somewhat saturated models. We shall state this as a conjecture below. 
\end{dsc}

\begin{con}
 Let $T$ be countable; then the class of $\aleph_1$-saturated models of $T$ 
is absolutely categorical over $P$ if and only if $T$ is nulldimensional, and every u.l.a. $\mathcal P^-(n)$-system is stable over $P$.
\end{con}

It would be interesting to know which ones of the implications in Corollary \ref{co:nullcatequiv} are reversible, and whether there are real gaps between the notions above. For example: is there a characterization of categoricity (or absolute categoricity) over $P$ for a countable theory $T$ using nulldimensionality and stability?

%

\bigskip
We conclude the paper with a few simple examples. 

\begin{ex}
\begin{enumerate}
\item

 Let $T$ be the complete theory of an infinite abelian group $G$, and $P$ a predicate picking out a subgroup. We have noted in the discussion preceding Corollary \ref{co:hodges}  that this example fits into our framework, and that our results (really, Lachlan's results) imply that $T$ is fully stable over $P$ and has full extension over $P$. 
 
 Now, if $P$ has finite index in $G$, then $T$ is categorical over $P$ (therefore nulldimensional). $T$ is generally not rigid over $P$, since the cosets can be permuted by an automorphism of $G$ while fixing $P$. 
 
 If $P$ has infinite index, $T$ is not nulldimensional (and not categorical): over every model $M$, there is a non-algebraic type weakly orthogonal to $P$, the type of a new coset element. On the other hand, the type of a new element in an existing coset is in $S_*(M)$. 
 
\item

Let $T$ be the theory of an equivalence relation $E$ with infinitely many infinite equivalence classes, $P$ picks exactly one element from each class. It is not hard to see that the hypotheses of this paper are satisfied. Like in the first example, $T$ is stable, hence is fully stable over $P$, and, by Lachlan's Theorem, has full existence. In fact, $T$ falls on the ``low'' side of all the classical classification theory dividing lines. Yet, given an uncountable model $N$ of $T^P$ (which is just an uncountable set), there are $2^{|N|}$ models of $M_i$ of $T$ (up to isomorphism fixing $N$) with  $P^{M_i} = N$. This is 
because every class of $E_0$ can be chosen to be either countable or uncountable (and an isomorphism over $P$ does not allow to exchange the classes). 

\item 

If, on the other hand, $Q = \cC\setminus P^\cC$ is our predicate in the previous example, then $T$ is categorical (hence unidimensional) over $Q$. 

\item

If, instead of an equivalence relation, we had a generic function from $Q$ onto $P$ (so every pre-image is infinite; or, rather, in order to fit the example into our framework, we replace the function with a binary relation $R$ defining a bipartite graph between $P$ and $Q$ with the appropriate properties), this would define an equivalence relation $E$ on $Q$ just as in Example (ii) above, and the conclusions would be the same. 

Considering $T$ over $Q$, however, poses a problem. It appears to be similar to Example (iii) above: e.g., there are no Vaughtian pairs over $Q$.  At the same time, the structure on $Q$ without the graph $R$ is just a set, and it is easy to see that for any infinite $\lam$ there are $2^\lam$ models over any model of $T^Q$. It is also easy to construct $2^\lam$ saturated models non-isomorphic over $Q$. This is because in this case Hypothesis \ref{asm:1} fails, so $T$ over $Q$ does not fall into our framework (and, as expected based on results in \cite{PiSh130}, this failure leads to non-structure). 
%

 \end{enumerate}

\end{ex}

\bibliography{common.bib}
\bibliographystyle{alpha}

\Addresses

\end{document}